\documentclass[11pt,a4paper]{article}

\usepackage[utf8]{inputenc}
\usepackage[T1]{fontenc}
\usepackage{amsmath}
\usepackage{amsfonts}
\usepackage{amssymb}
\usepackage{amsthm}
\usepackage{thmtools}
\usepackage{enumerate}
\usepackage{mathtools} 
\usepackage{stmaryrd}
\usepackage{xfrac}
\usepackage{lmodern}
\usepackage[nottoc]{tocbibind} 
\usepackage{hyperref} 
\usepackage[capitalise]{cleveref} 
\usepackage{graphicx}
\usepackage[english]{babel}
\usepackage{slashed}
\usepackage{microtype}
\usepackage[autostyle=true]{csquotes}
\usepackage[all,cmtip]{xy}
\usepackage{tikz-cd}
\usepackage{rotating}

\usepackage[a4paper,
            top=2.5cm,
            bottom=2.5cm,
            inner=3.5cm,
            outer=3.5cm]{geometry}

\usepackage[citestyle=alphabetic,
            bibstyle=alphabetic,
            backend=biber,
            url=true,
            doi=true,
            isbn=false,
            giveninits=true]{biblatex}

\AtEveryBibitem{\ifentrytype{article}{
    \clearfield{url}\clearfield{urldate}}{}
  \ifentrytype{book}{
    \clearfield{url}\clearfield{urldate}}{}
  \ifentrytype{incollection}{
    \clearfield{url}\clearfield{urldate}}{}
}
\usepackage{xparse} 
\usepackage{enumitem}
\setlist[enumerate]{label={\upshape(\roman*)}}
\newlist{myenumi}{enumerate}{1}
\setlist[myenumi,1]{label=\upshape(\roman*)}
\newlist{myenuma}{enumerate}{1}
\setlist[myenuma,1]{label=\upshape(\alph*)}

\declaretheorem[name=Theorem,numberwithin=section]{thm}
\declaretheorem[name=Lemma,numberlike=thm]{lem}
\declaretheorem[name=Corollary,numberlike=thm]{cor}
\declaretheorem[name=Proposition,numberlike=thm]{prop}

\declaretheorem[name=Question, numberlike=thm]{question}

\declaretheorem[name=Definition,numberlike=thm,style=definition,qed=\(\blacklozenge\)]{defn}
\declaretheorem[name=Example,numberlike=thm,style=definition,qed=\(\blacklozenge\)]{example}
\declaretheorem[name=Remark,numberlike=thm,style=definition,qed=\(\blacklozenge\)]{rem}

\declaretheorem[name=Assumption,numberlike=thm,style=definition,qed=\(\blacklozenge\)]{assumption}

\crefdefaultlabelformat{#2\textup{#1}#3}
\crefname{thm}{Theorem}{Theorems}
\crefname{lem}{Lemma}{Lemmas}
\crefname{defn}{Definition}{Definitions}
\crefname{prop}{Proposition}{Propositions}
\crefname{cor}{Corollary}{Corollaries}
\crefname{assumption}{Assumption}{Assumptions}
\crefname{equation}{}{}

\newcommand{\N}{\mathbb{N}}

\newcommand{\Z}{\mathbb{Z}}

\newcommand{\R}{\mathbb{R}}
\newcommand{\IR}{\mathbb{R}}
\newcommand{\C}{\mathbb{C}}

\newcommand{\supp}{\mathrm{supp}}

\newcommand{\cA}{\mathcal{A}}

\newcommand{\cS}{\mathcal{S}}

\newcommand{\cAz}{\cA_0}

\newcommand{\Hilbert}{\mathcal{H}}

\newcommand{\fA}{\mathfrak{A}}
\newcommand{\fm}{\mathfrak{m}}

\newcommand{\Ct}{\mathrm{C}}
\newcommand*{\Cz}{\Ct_0}

\newcommand*{\Cb}{\Ct_{\mathrm{b}}}
\newcommand*{\Cub}{\Ct_{\mathrm{ub}}}

\newcommand*{\Lin}{\mathfrak{B}}
\newcommand*{\Kom}{\mathfrak{K}}
\newcommand*{\grKom}{\widehat{\Kom}}

\newcommand*{\elltwo}{\ell^2}
\newcommand*{\grelltwo}{\widehat{\ell}^2}
\newcommand*{\Cl}{\C\ell}
\newcommand*{\ClR}{\mathrm{C}\ell}
\newcommand{\unitary}{\mathrm{U}}
\DeclareMathOperator{\SO}{SO}

\newcommand{\K}{\mathrm{K}}
\newcommand{\KO}{\mathrm{KO}}

\newcommand{\EE}{\mathrm{E}}

\newcommand{\frakc}{\mathfrak{c}}
\newcommand{\frakcbar}{\bar{\mathfrak{c}}}

\newcommand*{\sHigCom}{\frakcbar}
\newcommand*{\sHigCor}{\frakc}

\NewDocumentCommand{\textCstar}{}{\ensuremath{\mathrm{C}^*\!}}

\NewDocumentCommand{\LSym}{}{\mathrm{L}}

\NewDocumentCommand \RoeSymbol {o} {
	\mathrm{C}^{\ast}
	\IfNoValueF{#1}{_{#1}}
}

\NewDocumentCommand \VanishSymbol {o} {
	\mathrm{N}^{\ast}
	\IfNoValueF{#1}{_{#1}}
}

\NewDocumentCommand \FiproSymbol {o} {
	\mathrm{E}^{\ast}
	\IfNoValueF{#1}{_{#1}}
}

\NewDocumentCommand \RoePlaceholder {o} {
\RoeSymbol[
	\IfNoValueF{#1}{#1,}
	\mathrm{?}
]
}

\NewDocumentCommand \Roe {o} {\RoeSymbol[#1]}
\NewDocumentCommand \Fipro {o} {\FiproSymbol[#1]}

\NewDocumentCommand \varRoe {o} {
	\RoeSymbol[\sim\IfNoValueF{#1}{,#1}]
}

\NewDocumentCommand \Loc {o} {
	\RoeSymbol[
		\IfNoValueF{#1}{#1,}
		\LSym
	]
}

\NewDocumentCommand \LocVanish {o} {
	\VanishSymbol[
		\IfNoValueF{#1}{#1,}
    \LSym
	]
}

\NewDocumentCommand \FiproLoc {o} {
	\FiproSymbol[
		\IfNoValueF{#1}{#1,}
		\LSym
	]
}

\NewDocumentCommand \varLoc {o} {
	\RoeSymbol[
		\sim,
		\IfNoValueF{#1}{#1,}
		\LSym
	]
}

\NewDocumentCommand \Locz {o} {
	\RoeSymbol[
		\IfNoValueF{#1}{#1,}
		\LSym,0
	]
}

\NewDocumentCommand \FiproLocz {o} {
	\FiproSymbol[
		\IfNoValueF{#1}{#1,}
		\LSym,0
	]
}

\NewDocumentCommand \varLocz {o} {
	\RoeSymbol[
		\sim,
		\IfNoValueF{#1}{#1,}
		\LSym,0
	]
}

\newcommand*{\spinC}{spin\(^\mathrm{c}\)}

\newcommand*{\SpinC}{\mathrm{Spin}^\mathrm{c}}
\newcommand*{\Spin}{\mathrm{Spin}}

\DeclareMathOperator{\im}{im}
\newcommand{\id}{\mathrm{id}}

\DeclareMathOperator{\ind}{ind}
\DeclareMathOperator{\Ad}{Ad}

\DeclareMathOperator{\dist}{dist}
\DeclareMathOperator{\End}{End}
\DeclareMathOperator{\grad}{grad}
\DeclareMathOperator{\tw}{tw}
\DeclareMathOperator{\ev}{ev}

\NewDocumentCommand{\blank}{}{{-}}

\newcommand*{\grtensor}{\mathbin{\widehat{\otimes}}}
\newcommand*{\grexttensprod}{\mathbin{\widehat{\boxtimes}}}

\newcommand{\midd}{\mathrel{\|}}

\numberwithin{equation}{section} 

\author{Alexander Engel\thanks{Institut f\"{u}r Mathematik und Informatik, Universit\"at Greifswald,
Walther--Rathenau--Straße 47,
17489 Greifswald,
Germany\newline
\href{mailto:alexander.engel@uni-greifswald.de}{alexander.engel@uni-greifswald.de}}
\and Christopher Wulff\thanks{
Mathematisches Institut, Georg--August--Universität Göttingen, Bunsenstraße 3--5, 37073 Göttingen, Germany\newline
\href{mailto:christopher.wulff@mathematik.uni-goettingen.de}{christopher.wulff@mathematik.uni-goettingen.de}}
}

\title{The relative index in coarse index theory and submanifold obstructions to uniform positive scalar curvature}
\date{}

\begin{document}

\maketitle

\begin{abstract}
We provide a coarse version of the relative index of Gromov and Lawson and thoroughly establish all of its basic properties. As an application, we discuss a general procedure to construct wrong way maps on the $K$-theory of the Roe algebra mapping the coarse index class of the Dirac operator of a manifold to the one of a suitably embedded submanifold of arbitrary codimension,
thereby establishing an abstract machinery to find obstructions to uniform positive scalar curvature coming from these submanifolds.
\end{abstract}

\setcounter{tocdepth}{2}
\tableofcontents

\section{Introduction}

While generalized Dirac operators over compact manifolds are Fredholm operators and hence have a Fredholm index, this is no longer the case over non-compact complete Riemannian manifolds.
There are various ways in which one can deal with this problem if one still wants to do index theory, each with its own advantages and disadvantages. 
One of them is the relative index introduced by Gromov and Lawson in \cite[Section 4]{gromov_lawson_complete}: Suppose that $D_{1,2}$ are generalized Dirac operators acting on smooth sections of bundles $S_{1,2}$ over complete Riemannian manifolds $M_{1,2}$ and assume that they agree outside of a compact subset. That is, there exist compact subsets $L_1\subset M_1$ and $L_2\subset M_2$, an isometry $\psi\colon M_1\setminus L_1\to M_2\setminus L_2$ and a bundle isomorphism $\Psi\colon S_1|_{M_1\setminus L_1}\to S_2|_{M_2\setminus L_2}$ such that $D_2=\Psi\circ D_1\circ \Psi^{-1}$.
Then there are several ways to define a relative index $\ind(D_1,D_2)\in\Z$. 
It has the advantage that it is an easily accessible numerical invariant, which is due to the fact that it only contains ``compactly supported'' index information from the regions where the two operators differ and hence it can be deduced from classical Fredholm index theory.

A different approach is the coarse index theory introduced by John Roe \cite{roe_index_1,roe_index_2, roe_partitioning,roe_coarse_cohomology,roe_index_coarse,HigRoe}.
He observed that generalized Dirac operators $D$ on complete Riemannian manifolds $M$ are invertible modulo a certain tailor-made \textCstar-algebra $\Roe(M)$, nowadays called the Roe algebra, and hence one can define their coarse index as an element in $\K$-theory: $\ind(D)\in\K(\Roe(M))$.
It contains index information from the whole manifold, but it is in general less accessible, because it is not a numerical invariant.

The purpose of our work is to combine these two concepts and discuss relative coarse indices for situations where the two subsets $L_1,L_2$ are not compact.
It should be said that Roe already defined a relative coarse index in \cite[Section 4]{Roe_Relative}, although with different methods than we do. 
We have not checked whether they are equal (up to canonical maps), but considering that Gromov-Lawson's very important relative index theorem \cite[Theorem 4.18]{gromov_lawson_complete} has almost identical canonical generalizations in both Roe's work \cite[Theorem 4.6]{Roe_Relative} and ours (see \Cref{thm:PhiIndexTheorem}, which we will adress in a second), it seems highly likely that they are. A relative coarse index was also constructed in \cite[Section~10]{BE_relative_coarse_index}.

Before going any further into the details of our work, it should be mentioned that we use the set-up of \cite{WulffTwisted} as the foundation for coarse index theory, because we will also make very heavy use of the results therein. 
The operators under consideration are $A$-linear Dirac operators $D$ over complete Riemannian manifolds $M$, $A$ a graded unital \textCstar-algebra, and such operators have indices in $\K(C^*(M;A))$, the $\K$-theory of the Roe algebra with coefficients in $A$.
It has the advantage that we do not have to distinguish between the even dimensional case with graded self-adjoint operators and the odd dimensional case with ungraded self-adjoint operators, because that can be dealt with via Bott periodicity by using coefficients in Clifford algebras instead. We will also implement the grading of $\K$-theory via Clifford algebras, i.\,e.\ instead of $\K_p(A)$ we will usually write one of the canonically isomorphic groups $\K(A\grtensor \Cl_{r,s})$ with $r-s=p$.
The whole set-up will be recalled in \Cref{sec:recapitulation}.

Our approach to the relative coarse index is to follow the definition of the $\Phi$-relative index in \cite[(4.32) on page 339]{gromov_lawson_complete}.
In the situation described above, if $L_1=W_1$ and $L_2=W_2$ are two closed submanifolds of codimension zero with boundary, then $\Psi$ can be used to glue the two $A$-linear Dirac operators $D_1,D_2$ together to yield an $A$-linear Dirac operator $D_X$ over the complete manifold 
\[X\coloneqq W_1 \cup_{\psi|_{\partial W_1}\colon\partial W_1\to \partial W_2} W_2\]
and we define the relative coarse index
\[\ind(D_1,D_2\mid \Psi)\coloneqq \ind (D_X)\in\K(\Roe(X;A))\]
in \Cref{defn:PhiRelInd}.

Our first main theorem is our version of the relative index theorem.
To formulate it, we will introduce in \Cref{sec:invertibleawayfromclosedsubsets} the notion of an $A$-linear Dirac operator $D$ on $M$ which is ``invertible away from a closed subset $L$''. Primary examples are the spin Dirac operator on a complete Riemannian spin manifold that has uniformly positive scalar curvature outside of some $R$-neighborhood of $L$ and Fredholm operators if $L$ is compact. Such operators have an $L$-local coarse index 
\[\ind(D\mid L)\in \K(\Roe(L\subset M;A))\cong\K(\Roe(L;A))\,,\]
where $\Roe(L\subset M;A)\subset \Roe(M;A)$ is a certain well-known ideal with the same $\K$-theory as $\Roe(L;A)$.
Recalling that the $\K$-theory of the Roe algebra is functorial under coarse maps, we now have everything to understand the statement of the theorem.

\begin{thm}[\Cref{thm:PhiIndexTheorem}]
Assume that $D_1,D_2$ are invertible away from $W_1,W_2$, respectively.
Then 
\[\ind(D_1,D_2\mid\Psi)=(\iota_1)_*\ind(D_1\mid W_1)-(\iota_2)_*\ind(D_2\mid W_2)\]
where $\iota_{1,2}\colon W_{1,2}\to X$ denote the canonical inclusions.
\end{thm}

Subsequently we will restrict our attention to the special case where we have one $A$-linear Dirac operator $D$ over a complete Riemannian manifold $M$ twisted by two different bundles $E,F\to M$ which are isomorphic outside of some $R$-neighborhood $U$ of a closed subset $L\subset M$. The bundles $E,F$ can be more general than vector bundles, namely bundles whose fibres are finitely generated projective $B$ modules over another graded unital \textCstar-algebra $B$, $B$-bundles for short. Then the twisted operators $D_E,D_F$ agree outside of $U$ and we will obtain a relative index
\[\ind(D\midd E,F)\in \K(\Roe(L;A\grtensor B))\]
in \Cref{defn:TwistedRelInd}.

One of the main results of \cite{WulffTwisted} is that for certain $B$-bundles $E\to M$, the coarse index of the twisted operator can be calculated by a composition product
\[\ind(D_E)=\llbracket D;B\rrbracket \circ\llbracket E\rrbracket\]
in $\EE$-theory \cite[Theorem 4.14]{WulffTwisted}, see \Cref{thm:IndexTwistedOperator}. The condition that $E$ needs to satisfy is that it is given by a smooth projection $P\in\cA(M;B)$ where $\cA(M;B)\subset \Cb(M;B\grtensor\Kom)$ is the norm closure of the subalgebra of all bounded smooth functions $f\colon M\to B\grtensor\Kom$ with bounded gradient. Such bundles are called $B$-bundles of bounded variation. In the formula, $\llbracket E\rrbracket$ is the $\K$-theory class of the projection $P$ and $\llbracket D;B\rrbracket$ is an $\EE$-theory class in $\EE(\cA(M;B),\Roe(M;A\grtensor B))$ which is related to the usual $\K$-homology class of $D$.

Our next main result is that an analogous formula also holds for the relative index for twisted operators.
Assume that $E,F\to M$ are $B$-bundles of bounded variation and that the defining projections $P,Q\in\cA(M;B)$ agree outside of the $R$-neighborhood $U$.
Then
\[P-Q\in \cA(\overline{U},\partial U;B)\subset\cAz(L\Subset M;B)\subset \cA(M;B)\]
where $\cA(\overline{U},\partial U;B)$ denotes the ideal of all functions vanishing outside of $U$ and $\cAz(L\Subset M;B)$ is the ideal of functions which become arbitrary small outside of sufficiently large $R$-neighborhoods of $L$, that is, it is a generalization of the concept of $\Cz$-functions. The projections $P,Q$ therefore determine a $\K$-theory class $\llbracket E\rrbracket - \llbracket F\rrbracket \in \K(\cAz(L\Subset M;B))$ and one can also define an $\EE$-theory class $\llbracket D\midd L;B\rrbracket\in\EE(\cAz(L\Subset M;B),\Roe(L\subset M;A\grtensor B))$.

\begin{thm}[\Cref{thm:TwistedRelIndETheory}]
\(\ind(D\midd E,F)=\llbracket D\midd L;B\rrbracket\circ (\llbracket E\rrbracket-\llbracket F\rrbracket)\)
\end{thm}

In particular, this means that the relative index of twisted operators only requires the $\K$-theory class $\llbracket E\rrbracket-\llbracket F\rrbracket$ and not the bundles $E,F$ themselves. 
Just as $\K$-theory classes are an abstraction of vector bundles, $\K$-homology classes are an abstraction of elliptic operators and we will briefly complete this picture in \Cref{sec:twistingKtheory} by taking not the twist of the operator $D$ with the two bundles $E,F$ individually, but by twisting its $\K$-homology class $[D]$ with the $\K$-theory class $\llbracket E\rrbracket-\llbracket F\rrbracket$. The relative index of the twisted operators is then the same as an abstract index map applied to the twist of these two classes.

In order to obtain significant applications of the relative coarse index, we need suitable bundles to twist with, and luckily there are very interesting candidates: Consider a complete submanifold $L\coloneqq N\subset M$ with $\K$-oriented normal bundle and bundles $E,F$ which agree outside of a tubular neighborhood on $N$ and whose difference represents a Thom class $\tau$. It is a priori not clear what the latter should mean, but we will elaborate situations in which one can define $\tau$ meaningfully as an element of $\K(\cA(N\Subset M;\Cl_{0,r}))$ where $r$ denotes the codimension of $N$ in $M$.
Associated to the $\K$-orientation on $V$, i.\,e.\ a \spinC-structure, there is a spior $\Cl_r$-bundle $S_V$ and dual to it a spinor $\Cl_{0,r}$-bundle $S'_V$.
While we we will use $S'_V$ to construct the Thom class, $S_V$ will serve to relate an $A\grtensor\Cl_r$-linear Dirac operator $D_M$ on $M$ with an $A$-linear Dirac operator $D_N$ on $N$.
All of this will be discussed in far more detail in \Cref{sec:Thomclassrelation} and it will lead us to the following main theorem relating the relative coarse index $\ind(D_M\midd E,F)=\llbracket D_M\midd N;\Cl_{0,r}\rrbracket\circ \tau$ of $D_M$ twisted by the Thom class with the coarse index of $D_N$ mapped into $\K(\Roe(N\subset M;A))$ via functoriality under the inclusion map $i\colon N\to M$.

\begin{thm}[\Cref{thm_Thomclassrelation_relaxed}]
Let $M$ be a complete Riemannian manifold and $N\subset M$ a complete Riemannian submanifold with $\K$-oriented normal bundle $V$. We assume that $N$ is uniformly embedded into $M$, $M$ has bounded geometry in some $R$-neighborhood of $N$ and that the spinor $\Cl_{0,r}$-bundle $S'_V$ associated to $V$ has bounded geometry.
Furthermore we assume that $S_M|_N$ with its restricted action of $\Cl(TM)|_N\cong\Cl(TN)\grtensor\Cl(V)$ is isomorphic to $S_N\grtensor S_V$. Then:
\[\llbracket D_M\midd N;\Cl_{0,r}\rrbracket\circ \tau=i_*\ind(D_N)\]
\end{thm}

The relation between the relative index on $M$ and the index on $N$ opens up the possibility to use the index theory on the submanifold $N$ in order to study the geometry of $M$. We will show how on spin manifolds one can use $i_*\ind(D_N)$ as an obstruction to uniformly positive scalar curvature away from $N$, that is, outside of an $R$-neighborhood of $N$.

To this end we first study in \Cref{sec:indexawayfromL} the coarse index away from $L\subset M$ of an $A$-linear Dirac operator $D$ on $M$. It is the image $\ind(D,L)\in\K(\Roe(M,L;A))$ of the coarse index $\ind(D)\in\K(\Roe(M;A))$ under the canonical map induced by
\[\Roe(M;A)\to\frac{\Roe(M;A)}{\Roe(L\subset M;A)}\mathrel{{=}{:}}\Roe(M,L;A)\,,\]
where the quotient on the right hand side is called the relative Roe algebra.
It is actually not necessary for the operator $D$ to be defined over all of $M$ in order to construct this index. Instead it suffices if it is defined outside of some $R$-neighborhood of $L$, or away from $L$ as we had called it.

The index away from $L$ has interesting geometric consequences. Just like the normal coarse index of the spin Dirac operator $\slashed D$ over a complete Riemannian spin manifold $M$ is an obstruction to uniformly positive scalar curvature, its index away from $L$ is an obstruction to uniformly positive scalar curvature away from $L$, see \Cref{cor_upsc_outside_L}.

If in addition to the $A$-linear Dirac operator defined away from $L$ we are given a $B$-bundle $E$ defined away from $L$, then one can also consider the index away from $L$ of the twisted operator $D_E$.
We give two formulas to calculate this index for special types of bundles.
First, if $E$ has bounded variation, then its defining projection will yield a projection in the quotient \textCstar-algebra
\[\cA(M\div L;B)\coloneqq \cA(M;B)/\cAz(L\Subset M;B)\]
representing a $\K$-theory class $\llbracket E\div L\rrbracket$ of the bundle.
Furthermore we will define an $\EE$-theory class
\[\llbracket D, L;B\rrbracket\in \EE( \cA(M\div L;B),\Roe(M,L;A\grtensor B)) \]
of the operator which allows us to prove the following generalization of \cite[Theorem 7.10]{WulffTwisted}.

\begin{thm}[cf.\ \Cref{thm:indexawayfromLEtheoryproduct}]
Let $E$ be a $B$-bundle of bounded variation defined over the complement of a controlled neighborhood of $L$ in $M$.
Then the index away from $L$ of the twisted operator $D_E$ is given by the formula
\[
\ind(D_E,L)=\llbracket D,L;B\rrbracket\circ\llbracket E\div L\rrbracket\,.
\]
\end{thm}

Second, one can consider bundles of vanishing variation away from $L$, which are defined as follows. 
Let $\sHigCom_L(M;B)$ be the norm closure of the subalgebra of $\cA(M;B)$ consisting of the smooth functions whose gradient becomes arbitrary small outside of sufficently large $R$-neighborhoods of $L$. We call it the stable Higson compactification of $M$ away from $L$ with coefficients in $B$. It contains $\cAz(L\subset M;B)$ as an ideal and we define the stable Higson corona of $M$ away from $L$ as the quotient \textCstar-algebra
\[\sHigCor_L(M;B)\coloneqq \frac{\sHigCom_L(M;B)}{\cAz(L\Subset M;B)}\subset\cA(M\div L;B)\,.\]
These are direct generalizations of the \textCstar-algebras introduced by Emerson and Meyer in \cite{EmeMey}.

Now, a $B$-bundle of vanishing variation away from $L$ is a $B$-bundle which is represented by a smooth function in $\sHigCom_L(M;B)$ that is projection valued away from $L$. Such bundles clearly have a class $\llbracket E,\sHigCor_L\rrbracket\in\K(\sHigCor_L(M;B))$ represented by this function.
The surprising fact is that here do not need to know the whole $\EE$-theory class of $D$ in order to calculate the index $\ind(D_E,L)$ of the twisted operator $D_E$, but instead it is sufficient to know the index $\ind(D,L)$.
In direct generalization of \cite[Theorem~8.7]{WulffTwisted} it can be calculated using a cap product
\[\cap\colon\K(\Roe(M,L;A))\otimes\K(\sHigCor_L(M;B))\to\K(\Roe(M,L;A\grtensor B))\]
as follows.

\begin{thm}[\Cref{thm:indexcapproduct}]
Let $M$ be a complete Riemannian manifold of bounded geometry, $D$ an $A$-linear Dirac operator defined on the complement of some controlled neighborhood of a closed subset $L\subset M$ and $A,B$ graded \textCstar-algebras.
Letting $\iota\colon\sHigCor_L(M;B)\hookrightarrow\cA(M\div L;B)$ denote the inclusion, then
\begin{equation*}
\llbracket D,L;B\rrbracket \circ\iota_*(y)=\ind(D,L)\cap y
\end{equation*}
for all $y\in \K(\sHigCor_L(M;B))$.
In particular, for every $B$-bundle $E$ of bounded variation away from $L$ the index of the twisted operator $D_E$ can be calculated by
\[\ind(D_E,L)=\ind(D,L)\cap\llbracket E,\sHigCor_L\rrbracket\,.\qedhere\]
\end{thm}

By combining the previous constructions and results we obtain the following theorem, which establishes an abstract machinery to find submanifold obstructions to positive scalar curvature.

\begin{thm}[\Cref{thm_thom_lifts_application}]
Let $M$ be a complete Riemannian $m$-dimensional spin-manifold and $N\subset M$ be a complete $n$-dimensional submanifold with $\K$-oriented normal bundle~$V$. Hence $N$ is at least \spinC\ and we let $\slashed D_M$ and $\slashed D_N$ be the $\Cl_m$- and $\Cl_n$-linear spin (resp. \spinC) Dirac operators of $M$ and $N$, respectively. Let $i\colon N\hookrightarrow M$ denote the inclusion and $r\coloneqq m-n$ the codimension.

Assume that $N$ is uniformly embedded into $M$, that $M$ has bounded geometry in some $R$-neighbourhood of $N$ and that the line bundle associated to $P_{\unitary_1}(V)$ of the \spinC -structure of $N$ has bounded geometry (which is the case if $V$ is $\KO$-oriented).

If the Thom class $\tau$ lies in the image of the coarse co-assembly map away from $N$
\[\mu_N\colon \K(\sHigCor_N(M;\Cl_{1,r}))\to\K(\cAz(N\Subset M;\Cl_{0,r}))\]
with preimage $\tilde\tau$,
then 
the wrong way map 
\begin{align*}
\K(\Roe(M,N;\Cl_m))&\to  \K(\Roe(N\subset M;\Cl_{m,r}))\\
x&\mapsto \partial(x\cap\tilde\tau)
\end{align*}
maps $\ind(\slashed D_M, N)$ to $i_*\ind(\slashed D_N)$.
Thus, $i_*\ind(\slashed D_N)\not=0$ implies that
 there cannot be a metric of uniformly positive scalar curvature away from $N$ in the same quasi-isometry class as the original one.
\end{thm}

The coarse co-assembly map away from $N$ that appears in this theorem is a direct generalization of the co-assembly map introduced by Emerson and Meyer in \cite{EmeMey}.
In order to get this abstract machinery started, the important question is as follows.
\begin{question}\label{question}
Under which conditions does the Thom class lie in the image of the coarse co-assembly map away from $N$? 
\end{question}

As our final result, \Cref{thm_multi_partitioned} answers the question positively for multi-partitioned manifolds, thereby recovering results of Schick and Zadeh \cite{schick_zadeh} and \cite[Section~4.3]{siegel_analytic_structure_group}. In these references and also in \cite{bunke_ludewig_callias}, however, the wrong way map is obtained differently, namely by iterated application of a Mayer--Vietoris boundary map.
The relative coarse index was also put into connection with partitioned manifolds in \cite{KZZ}, although in a different manner.

Further instances where \Cref{question} has an affirmative answer will have to be investigated in future research.

\paragraph{Acknowledgements.}
The authors would like to thank Thomas Schick for the inspiration to think about relative coarse indices and Nigel Higson as well as Thorsten Hertl for answering their questions.

Alexander Engel and Christopher Wulff acknowledge financial support by the DFG through the Priority Programme SPP 2026 ``Geometry at Infinity'', project numbers 441426261 and 338480246.

\section{Recapitulation of coarse index theory}
\label{sec:recapitulation}

In this section we recall some basic notions and results of coarse index theory of $A$-linear Dirac operators, $A$ a graded unital \textCstar-algebra. We adopt the set-up of the article \cite{WulffTwisted}, because we will make heavy use of it. In order to keep our presentation concise, we will recall only the most important parts and refer to this reference for further details.

In the following, $A,A_1,A_2,B,C$ will always denote graded, unital, complex \textCstar-algebras.
The graded \textCstar-algebras of bounded and compact operators on a graded right Hilbert $A$-module $\Hilbert$ will be denoted by $\Lin_A(\Hilbert)$ and $\Kom_A(\Hilbert)$, respectively, or $\Lin(\Hilbert)$, $\Kom(\Hilbert)$ if $\Hilbert$ is a graded Hilbert space.
Let $\elltwo\coloneqq\elltwo(\N)$ be our standard choice of separable infinite dimensional Hilbert space and $\grelltwo\coloneqq\elltwo\oplus\elltwo$ the graded Hilbert space with even and odd part isomorphic to $\elltwo$. The associated \textCstar-algebras of compact operators on these Hilbert spaces will be denoted by $\Kom\coloneqq\Kom(\elltwo)$ and $\grKom\coloneqq\Kom(\grelltwo)$.

The letter $\cS$ will denote the \textCstar-algebra $\Cz(\R)$ but equipped with the grading in even and odd functions. The Clifford algebras of $\R^{r+s}$ equipped with the standard quadratic form of signature $(r,s)$ is denoted by $\Cl_{r,s}$.
Furthermore, the symbol $\grtensor$ will be used for the \emph{maximal} graded tensor product of graded \textCstar-algebras.

Finally, $X,X_1,X_2$ will usually denote proper metric spaces, except for a few cases where we use the letter $X$ for vector fields.
We allow the metrics to take the value infinity, because our main spaces of interest are possibly disconnected complete Riemannian manifolds, whose path-metric is of this type. By a \emph{controlled neighborhood} of a subspace $L\subset X$ we mean a neighborhood that is contained in the $R$-neighborhood of $L$ for some $R>0$.

\subsection{The Roe algebra with coefficients}

Roe algebras with coefficients in \textCstar-algebras have been considered previously in, for example, \cite{higson_pedersen_roe,hanke_pape_schick,WulffTwisted,WulffEquivariant} and work in most aspects exactly as in the classical case without coefficents, see \cite[Chapter 6]{HigRoe}. 

\begin{defn}[{\cite[Definition 2.2]{WulffTwisted}, see also \cite[Definition 3.2]{hanke_pape_schick}, \cite[Definitions 5.1--5.3]{higson_pedersen_roe}}]
Let $\Hilbert$ be a separable right Hilbert $A$-module,  $\rho\colon C_0(X)\to\Lin_A(\Hilbert)$ a representation and $T\in\Lin_A(\Hilbert)$.
\begin{itemize}
\item $T$ is \emph{locally compact} if $T\circ\rho(f),\rho(f)\circ T\in\Kom_A(\Hilbert)$ for all $f\in C_0(X)$
\item $T$ has \emph{finite propagation} if there exists $R>0$ such that $\rho(f)T\rho(g)$ vanishes for all $f,g\in C_0(X)$ with $\dist(\supp(f),\supp(g))\geq R$. The smallest such $R$ is called the \emph{propagation} of $T$. 
\item The \emph{Roe-\textCstar-algebra} of $X$ associated with $\rho$ is the sub-\textCstar-algebra $\Roe(X;\Hilbert)\subseteq \Lin_A(\Hilbert)$ generated by all locally compact operators with finite propagation. Its dependence of the representation $\rho$ is understood implicitly.\qedhere
\end{itemize}
\end{defn}

If for each space $X$ a sufficiently large Hilbert $A$-modules $\Hilbert_X$ is chosen, the usual theory of covering isometries goes through, turning $X\mapsto\K(\Roe(X;\Hilbert_X))$ into a functor. We shall not recall it here, but refer to the references above, in particular \cite[Section 4.4]{WulffEquivariant} for a more extensive treatment of the homological properties but also \cite[Definition 2.3--Lemma 2.7]{WulffTwisted} and \cite[Proposition 5.5]{higson_pedersen_roe}.

The most important spaces for us are complete Riemannian manifolds or closures of open subsets therein.
In this case, a convenient choice of sufficiently large Hilbert module is $L^2(X)\grtensor A\grtensor\widehat{\ell^2}$. 
\begin{defn}[{cf.\ \cite[Definition 2.6]{WulffTwisted}}]\label{def:concreteRoealgebra}
We define $\Roe(X;A)$ to be the Roe algebra of $X$ with respect to a sufficiently large Hilbert $A$-module. More concretely, if $X$ is the closure of an open subset of a complete Riemannian manifold and $A$ a graded unital \textCstar-algebra, then the Roe algebra of $X$ with coefficients in $A$ is defined as
\[
\Roe(X;A)\coloneqq \Roe(X;L^2(X)\grtensor A\grtensor\widehat{\ell^2}))\,.\qedhere
\]
\end{defn}

In the context of index theory, however, it is often more appropriate to choose Hilbert modules based on the space of $L^2$-sections of bundles $S\to X$.
\begin{defn}[{\cite[Definition 2.8]{WulffTwisted}}]
Let $X$ be the closure of an open subset of a complete Riemannian manifold. \begin{itemize}
\item By an \emph{$A$-bundle} $S\to X$ we mean a smooth bundle whose fibres are finitely generated, projective, graded, right Hilbert $A$-modules.
\item The Hilbert $A$-module $L^2(S)$ or $L^2(X,S)$ is the completion of the $A$-module $\Gamma_{\operatorname{cpt}}(S)$ of smooth compactly suported sections of $S$ with respect to the $A$-valued scalar product
\[(\xi,\zeta)=\int_M\langle\xi(x),\zeta(x)\rangle d\!\operatorname{vol}\,,\]
where $\langle.,.\rangle\colon \Gamma(S)\times\Gamma(S)\to C^\infty(X;A)$ is the fibrewise $A$-valued inner product.
\item Furthermore, we define $\Roe(X;S)\coloneqq \Roe(X;L^2(S)\grtensor\grelltwo)$.\qedhere
\end{itemize}
\end{defn}

Note that $L^2(S)\grtensor\grelltwo$ might not be sufficiently large even if $S$ is not the zero bundle, because finitely generated projective modules can disregard whole parts of a \textCstar-algebra. Nevertheless, one can always embed $S$ as a subbundle into the trivial bundle $(A\grtensor\grelltwo)\times X\to X$ and thus obtain isometric inclusions of graded Hilbert $A$-modules 
\[L^2(S)\hookrightarrow L^2(S)\grtensor\grelltwo\hookrightarrow L^2(X)\grtensor A\grtensor\grelltwo\grtensor\grelltwo\cong L^2(X)\grtensor A\grtensor\grelltwo\,.\]
By adjoing with them one obtains inclusions of graded \textCstar-algebras
\begin{equation*}\Roe(X;L^2(S))\hookrightarrow \Roe(X;S)\hookrightarrow \Roe(X;A)\,.
\end{equation*}
The maps on $\K$-theory induced by them do not depend on the choice of the embeddings.

Finally, in order to deal with the degree of $\K$-theory by means of Clifford algebras, it is important to make the following observation.
\begin{lem}[{cf.\ \cite[Definition 2.9]{WulffTwisted}}]\label{lem:RoeClifford}
There are canonical isomorphisms
\begin{align*}
\Roe(X;\Hilbert\grtensor\Cl_{r,s})&\cong \Roe(X;\Hilbert)\grtensor\Cl_{r,s}
\end{align*}
which are natural under the functoriality provided by covering isometries. \qed
\end{lem}

\subsection{Dirac operators over \texorpdfstring{\textCstar}{C*}-algebras}

Now let $M$ be a complete Riemannian manifold, possibly with boundary. In this subsection we specify what we mean by $A$-linear Dirac operators over $M$ and recall some basic constructions.

\begin{defn}[{\cite[Definitions 3.1--3.3]{WulffTwisted}}]
\begin{enumerate}
\item Let $S\to M$ be a smooth $A$-bundle. A \emph{connection} on $S$ is a $\C$-linear map $\nabla\colon \Gamma(S)\to\Gamma(S\otimes T^*M)$ which is grading preserving, i.\,e.\ maps $\Gamma(S^\pm)$ to $\Gamma(S^\pm\otimes T^*M)$, satisfies the Leibniz rule $\nabla_X(\xi\cdot f)=\nabla_X(\xi)\cdot f+\xi\cdot \partial_Xf$
for all sections $\xi\in\Gamma(S)$, all smooth $A$-valued functions $f\in C^\infty(M;A)$ and all $X\in TM$, and is metric with respect to the $A$-valued inner product $\langle.,.\rangle$ on the fibres, i.\,e.\ 
$\partial_X\langle\xi,\zeta\rangle=\langle\nabla_X\xi,\zeta\rangle+\langle \xi,\nabla_X\zeta\rangle$ for all $\xi,\zeta\in\Gamma(S)$ and $X\in TM$.

\item 
Let $\nabla^{\mathrm{LC}}$ denote the Levi--Civita connection on $M$.
A \emph{Dirac $A$-bundle} is a smooth $A$-bundle along with 
\begin{itemize}
\item a connection $\nabla\colon \Gamma(S)\to\Gamma(S\otimes T^*M)$;
\item a Clifford multiplication 
$\Cl(TM)\to\End_A(S)$, i.\,e.\ a graded \textCstar-algebra bundle homomorphism;
\end{itemize}
such that connection and Clifford multiplication are related by the Leibniz rule 
$\nabla_X(Y\cdot\xi)=(\nabla^{\mathrm{LC}}_XY)\cdot\xi+Y\cdot\nabla_X\xi$
for all $X\in TM$, $Y\in\Gamma(\Cl(TM))$ and $\xi\in\Gamma(S)$.

\item The \emph{Dirac operator} associated to a Dirac $A$-bundle $S\to M$ is the $A$-linear first order differential operator $D\colon \Gamma_{\operatorname{cpt}}(S)\to\Gamma_{\operatorname{cpt}}(S)$ which is locally given by
\[D=\sum_{i=1}^ne_i\nabla^S_{e_i}\,,\]
where $e_1,\dots, e_n$ is a local orthonormal frame of $TM$.
We call it for short an \emph{$A$-linear Dirac operator}.
\qedhere
\end{enumerate}
\end{defn}

Essential analytic properties of $A$-linear Dirac operators have been worked out in \cite{hanke_pape_schick,Ebert_AnalyticalFoundations}, in particular the following:

\begin{thm}[{\cite[Theorems 3.4,3.5]{WulffTwisted}, cf.\ \cite[Theorem 2.3 \& Lemma~3.6]{hanke_pape_schick} and \cite[Theorem 1.14]{Ebert_AnalyticalFoundations}}]
A Dirac operator $D$ over a complete Riemannian manifold $M$ without boundary\footnote{The condition that $M$ does not have a boundary was forgotten in the formulation of \cite[Theorems 3.4,3.5]{WulffTwisted}.} 
is closable in $L^2(S)$ and the minimal closure is regular  and self-adjoint as unbounded Hilbert-$A$-module operator.  It is the unique self-adjoint extension of $D$ and we denote it by the same letter.
The functional calculus yields a graded $*$-homomorphism
\[\cS\to \Roe(M;L^2(S))\,,\quad f\mapsto f(D)\]
with the property that $f(D)=Dg(D)=g(D)D$ if $f,g\in\cS$ satisfy $f(t)=tg(t)$ for all $t$. 
\qed
\end{thm}

Recall that in the spectral picture of $\K$-theory \cite{Trout}, elements of $\K(A)$ are homotopy classes of graded $*$-homomorphisms $\cS\to A\grtensor \grKom$. In particular, $*$-homomorphisms $\cS\to A$ define classes by tensoring them with a rank one projection in $\Kom\subset\grKom$. The coarse index has a very simple definition in this picture.

\begin{defn}[{\cite[Definitions 3.6]{WulffTwisted}}]\label{defn:coarseindex}
The coarse index $\ind(D)\in \K(\Roe(M;A))$ of $D$ is the $K$-theory class obtained by composing the $*$-homomorphism of the previous theorem with the inclusion $\Roe(M;L^2(S))\subseteq \Roe(M;A)$.
\end{defn}

We can also consider the class in $\K(\Roe(M;L^2(S)))$ or $\K(\Roe(M;S))$ as the coarse index, but we note that for full functoriality under coarse maps we need to work with the groups $\K(\Roe(M;A))$. 

The construction of twists and external products of Dirac operators works exactly as in the classical case without coefficients.

\begin{lem}[{\cite[Lemmas 3.7 \& 3.9]{WulffTwisted}}]\label{lem:twistedexternalConnection}
\begin{enumerate}
\item Let $S\to M$ be a $A$-bundle with connection $\nabla^S$ and $E\to M$ a $B$-bundle with connection $\nabla^E$. Then there is a unique connection $\nabla^{S\grtensor E}$ on the $A\grtensor B$-bundle $S\grtensor E$ such that 
\begin{equation}\label{eq:twistedConnection}
\nabla^{S\grtensor E}(\xi\grtensor\zeta)= \nabla^S\xi\grtensor\zeta+\xi\grtensor\nabla^E\zeta
\end{equation}
for all section $\xi\in\Gamma(S)$ and $\zeta\in\Gamma(E)$.

If $S$ is even a Dirac $A$-bundle, then $S\grtensor E$ is a  Dirac $A\grtensor B$-bundle with this connection and the obvious Clifford action. 

\item Let $S_i\to M_i$ be $A_i$-bundles with connections $\nabla^{S_i}$ for $i=1,2$. Then there is a unique connection $\nabla^{S_1\grexttensprod S_2}$ on the $A_1\grtensor A_2$-bundle $S_1\grexttensprod S_2\to M_1\times M_2$ such that 
\begin{equation*}\nabla^{S_1\grexttensprod S_2}_{X_1+X_2}(\xi_1\grexttensprod\xi_2)= \nabla^{S_1}_{X_1}\xi_1\grexttensprod\xi_2+ \xi_1\grexttensprod\nabla^{S_2}_{X_2}\xi_2
\end{equation*}
for all section $\xi_i\in\Gamma(S_i)$ and $X_i\in TM_i$.

If $S_{1,2}$ are even Dirac $A_{1,2}$-bundles, then $S_1\grexttensprod S_2$ is a  Dirac $A_1\grtensor A_2$-bundle with this connection and the obvious Clifford action. \qed
\end{enumerate}
\end{lem}

\begin{defn}[{\cite[Definitions 3.8 \& 3.10]{WulffTwisted}}]\label{def:twistedoperator}
\begin{enumerate}
\item Let $S\to M$ be a Dirac $A$-bundle, $D$ its associated Dirac operator and $E\to M$ a $B$-bundle with connection. Then the \emph{twisted operator} $D_E$ is the Dirac operator associated to the Dirac $A\grtensor B$-bundle $S\grtensor E$.
\item Let $S_i\to M_i$ be $A_i$-bundles with associated Dirac operators $D_i$ for $i=1,2$.
Then the \emph{external tensor product} $D_1\times D_2$ is the Dirac operator of the Dirac $A_1\grtensor A_2$-bundle $S_1\grexttensprod S_2\to M_1\times M_2$.
\qedhere
\end{enumerate}
\end{defn}

\begin{lem}\label{lem:functionalcalculusproductformula}
The functional calculus $\phi\mapsto \phi(D_1\times D_2)$
is equal to the composition
\begin{alignat*}{2}
\cS\xrightarrow{\Delta}\cS&\grtensor\cS&&\to\Roe(M_1;L^2(S_1))\grtensor \Roe(M_2;L^2(S_2))\subset \Roe(M_1\times M_2;L^2(S_1\grexttensprod S_2))
\\\phi_1&\grtensor\phi_2&&\mapsto\phi_1(D_1)\grtensor\phi_2(D_2)\,.
\end{alignat*}
\end{lem}

\begin{proof}
This follows directly from \cite[Lemma 3.1]{WulffTwisted}. \end{proof}

\subsection{The \texorpdfstring{$\EE$}{E}-theory class of an operator}
\label{sec:EEtheoryclass}

Dirac operators over compact manifolds $M$ have a $\K$-homology class $[D]\in\K(M)$ such that the index of the twist of $D$ with a vector bundle $E\to M$ can be calculated by the pairing between $\K$-homology and $\K$-theory: $\ind(D_E)=\langle [D],[E]\rangle\in\Z$. 
For $A$-linear Dirac operators over complete Riemannian manifolds $M$ twisted by $B$-bundles, the generalization of this formula is given by composition with an $\EE$-theory class.
Its definition involves a certain function algebra and the index formula works a priori for a special class of bundles, both of which are given in the following definition.

\begin{defn}[{cf.\ \cite[Definitions 4.3 \& 4.4]{WulffTwisted}}]
\begin{enumerate}
\item For a complete Riemannian manifold $M$, possibly with boundary, we define $\cA^\infty(M;B)$ to be the algebra of bounded smooth functions $M\to B\grtensor \grKom$ whose gradient is bounded. Let $\cA(M;B)$ be its norm-completion to a sub-\textCstar-algebra of $\Cb(M;B\grtensor\grKom)$.
\item A $B$-bundle $E\to M$ is called a \emph{$B$-bundle of bounded variation} if its fibers are the finitely generated projective Hilbert $B$-modules $E_x\coloneqq \im(P(x))\subseteq B\grtensor\grelltwo$ for a evenly graded projection valued function $P\in\cA^\infty(M;B)$. Its $\K$-theory class $\llbracket E\rrbracket\in \K(\cA(M;B))$ is the class of the projection $P$.
\qedhere
\end{enumerate}
\end{defn} 

Even though the given embedding of the bundle $E$ into the trivial bundle $(B\grtensor\grelltwo)\times M$ convenientely determines a connection, which we use to define the twisted Dirac operator $D_E$, bordism invariance (which will also be recalled below) implies that the coarse index of the latter does not depend on it. From this point of view it would be sufficient to consider the bundle $E$ up to isomorphism.
On the other hand, seeing the defining projection $P$ as part of the data has the advantage of unambiguously determining the $\K$-theory class and furthermore we can simply say that two such bundles agree on/outside of a subset of $M$ if they are truely equal, i.\,e.\ if the two projections take the same values there. That is why we choose the latter option.

The key observation is now \cite[Lemma 4.5]{WulffTwisted}: If $M$ is a complete Riemannian manifold without boundary and $S\to M$ is a Dirac-$A$-bundle and $D$ its associated $A$-linear Dirac operator, then
\begin{align}
\cS\grtensor \cA(M;B)&\to\fA( \Roe(M;S\grtensor B))\coloneqq \frac{\Cb([1,\infty); \Roe(M;S\grtensor B))}{\Cz([1,\infty); \Roe(M;S\grtensor B))} \nonumber
\\\phi\grtensor f&\mapsto\left[t\mapsto (\phi(t^{-1}D)\grtensor \id_{B\grtensor\grKom}) \cdot (\id_S\grtensor f)\right] \label{eq:DiracEclassAsympMorph}
\end{align}
is an asymptotic morphism. It is convenient to abbreviate the canonical actions of $\phi(t^{-1}D)\grtensor \id_{B\grtensor\grKom}$ and $\id_S\grtensor f$ on the Hilbert $A\grtensor B$-module $L^2(M,S)\grtensor B\grtensor\grelltwo\cong L^2(M,S\grtensor B\grtensor \grelltwo)$ simply by  $\phi(t^{-1}D)$ and $f$, respectively.

\begin{defn}[{\cite[Definition 4.6]{WulffTwisted}}]\label{def:ETheoryClassBoundedGradient}
The asymptotic morphism \eqref{eq:DiracEclassAsympMorph} determines an $E$-theory class
\[\llbracket D;B\rrbracket\in \EE(\cA(M;B), \Roe(M;S\grtensor B))\,.\]
We denote the image of this class in $\EE(\cA(M;B), \Roe(M;A\grtensor B))$ by the same symbol and call it the \emph{$\EE$-theory class of $D$ with respect to $B$}.
\end{defn}

The following theorem generalizes the index pairing.

\begin{thm}[{\cite[Theorem 4.14]{WulffTwisted}}]\label{thm:IndexTwistedOperator}
Let $M$ be a complete Riemannian manifold, 
$D$ a Dirac operator over $M$
and $E\to M$  a $B$-bundle of bounded variation.
Then the coarse index of the twisted operator $D_E$ is
\[
\pushQED{\qed}
\ind(D_E)=\llbracket D;B\rrbracket\circ \llbracket E\rrbracket\,.\qedhere
\popQED
\]
\end{thm}

We also need a version of the $\EE$-theory class that is localized over open subsets $U\subset M$, which we define now:
\begin{defn}
Let $\cA(\overline{U},\partial U;B)=\cA(M,M\setminus U;B)\subset \cA(M;B)$ be the sub-\textCstar-algebra of functions vanishing outside of $U$. 
\end{defn}
Note that $\cA(\overline{U},\partial U;B)$ does not depend on the ambient manifold $M$ into which it is Riemann-isometrically embedded, although the closure $\overline{U}=\overline{U}^{M}$ and and the boundary $\partial U=\partial_MU$ do. Furthermore, let $\overline{U}^{\subseteq M}$ denote the closure of $U$ equipped with the restriction of the path-metric on $M$, which, of course, also depends on the ambient manifold.
With the canonical isometric inclusion of Hilbert $A\grtensor B$-modules $V\colon L^2(\overline{U};S|_{\overline{U}}\grtensor B)\grtensor\grelltwo\to L^2(M;S\grtensor B)\grtensor\grelltwo$ there is an asymptotic morphism
\begin{align}
\cS\grtensor \cA(\overline{U},\partial U;B)&\to\fA( \Roe(\overline{U}^{\subseteq M};S|_{\overline{U}}\grtensor B)) \nonumber
\\\phi\grtensor f&\mapsto [t\mapsto V^*\phi(t^{-1}D)fV] \label{eq:DiracEclassAsympMorphSubset}
\end{align}
analogous to \eqref{eq:DiracEclassAsympMorph}, see \cite[Lemma 4.7]{WulffTwisted}.

\begin{defn}[{\cite[Definition 4.8]{WulffTwisted}}]\label{defn:restrictedEtheoryclass}
The asymptotic morphism \eqref{eq:DiracEclassAsympMorphSubset} defines the \emph{$E$-theory class of $D|_U$}
\[
\llbracket D|_{U};B\rrbracket\in \EE(\cA(\overline{U},\partial U;B),\Roe(\overline{U}^{\subseteq M};S|_{\overline{U}}\grtensor B))\,.\qedhere
\]
\end{defn}

This class satisfies
the following basic property.

\begin{lem}[Independence of extension {\cite[Theorem 4.9]{WulffTwisted}}] \label{lem:independence}
Assume that $U$ embedds Riemann-isometrically into two complete Riemannian manifolds $M_1,M_2$ and that $D_1,D_2$ are $A$-linear Dirac operators over $M_1,M_2$ which restrict to the same operator over $U$: $D_1|_U=D_2|_U$.
If the path-metric of $M_1$ restricted to $U$ is greater or equal to the path metric of $M_2$ restricted to $U$, then there is a commutative diagram
\[
\xymatrix@C=2cm{
\cA(\overline{U}^{M_1},\partial_{M_1} U;B)\ar[r]^-{\llbracket D_1|_U;B\rrbracket}\ar@{=}[d]
&\Roe(\overline{U}^{\subseteq M_1};S\grtensor B)\ar[d]^{\substack{\text{inclusion of \textCstar-algebras}\\=(\id\colon \overline{U}^{\subseteq M_1}\to \overline{U}^{\subseteq M_2})_*}}
\\\cA(\overline{U}^{M_2},\partial_{M_2} U;B)\ar[r]_-{\llbracket D_2|_U;B\rrbracket}
&\Roe(\overline{U}^{\subseteq M_2};S\grtensor B)
}
\]
in $\EE$-theory.\qed
\end{lem}

The most important use-case of the $\EE$-theory class localized over open subsets is if $U$ is the interior of a bordism, and indeed it becomes relevant in the proof of the following theorem.
For any interval $I$ we let $S_I$ denote the trivial Dirac bundle with fiber $\Cl_1$ over $I$ and  $D_I$ denote its Dirac operator. 
\begin{thm}[Bordism invariance, {\cite[Theorem 4.12]{WulffTwisted}}]\label{thm:BordismInvariance}
Let $W$ be a complete Riemannian manifold with a boundary $\partial W$ which decomposes into two complete Riemanian manifolds $M_0$, $M_1$ and assume that there are collar neighborhoods $M_0\times[0,\varepsilon)$ of $M_0$ and $M_1\times(-\varepsilon,0]$ of $M_1$ on which the Riemannian metric of $W$ is the product metric.
Let $D_0,D_1$ be $A$-linear Dirac operators over $M_0,M_1$, respectively, and assume that there is an $A\grtensor\Cl_1$-linear Dirac operator over $W$ which is equal to $D_0\times D_{[0,\varepsilon)}$ and $D_1\times D_{(-\varepsilon,0]}$ over the respective collar neighborhoods.
Then for every graded unital \textCstar-algebra $B$ the diagram
\[\xymatrix@C=3em{
&\cA(M_0;B)\ar[r]_-{\llbracket D_0;B\rrbracket}&\Roe(M_0;A\grtensor B)\ar[dr]_{(M_0\subseteq W)_*}&
\\\cA(W;B)\ar[ur]_{\text{restr.}}\ar[dr]^{\text{restr.}}
&&&\Roe(W;A\grtensor B)
\\&\cA(M_1;B)\ar[r]^-{\llbracket D_1;B\rrbracket}&\Roe(M_1;A\grtensor B)\ar[ur]^{(M_1\subseteq W)_*}&
}\]
commutes in $\EE$-theory.
\qed
\end{thm}

Composing the diagram in the theorem with the canonical element $1_W\in \K(\cA(W;\C))$ of the trivial rank-one bundle over $W$ implies that the coarse index is also bordism invariant, generalizing \cite{WulffBordism} to $A$-linear Dirac operators. 
Furthermore, the theorem implies that 
the $\EE$-theory class $\llbracket D;B\rrbracket\in \EE(\cA(M;B),\Roe(M;A\grtensor B))$ depends only on the principle symbol of $D$, i.\,e.\ it is independent of the choice of the Dirac connection on the bundle $S$, see \cite[Corollary 4.13]{WulffTwisted}.

\subsection{Localization algebras and the \texorpdfstring{$\K$}{K}-homology class of an operator}
\label{sec:localizationalgebrasKhomology}

Closely related to the $\EE$-theory class of an opertor is its $\K$-homology class when defined via localization algebras, which is an alternative approach to coarse index theory that was discussed quite heavily in \cite{zeidler_secondary,WilletYuHigherIndexTheory}. 
Here we will briefly recall one version of the localization algebra and the relation to the $\EE$-theory class developed in \cite[Section 5]{WulffTwisted}, but we reproduce it with more general coefficient \textCstar-algebras $A$ than just the Clifford algebras.

\begin{defn}
\begin{itemize}
\item The \emph{localization algebra} $\Loc(X;\Hilbert_X)$ of a proper metric space $X$ with respect to a representation of $\Cz(X)$ on a Hilbert $A$-module $\Hilbert_X$ is sub-\textCstar-algebra of $\Cb([1,\infty);\Roe(X;\Hilbert_X))$ generated by the uniformly continuous functions $L\colon [1,\infty)\to\Roe(X;A)$ such that the propagation of $L_t\coloneqq L(t)$ is finite for all $t$ and tends to zero as $t\to \infty$.

As before, we denote it by $\Loc(X;A)$ if a sufficiently large Hilbert $A$-module $\Hilbert_X$ has been chosen. 
In particular, if $X$ is the closure of an open subset of a complete Riemannian manifold, we choose $\Hilbert_X\coloneqq L^2(X)\grtensor A\grtensor\grelltwo$.

We define the \emph{$\K$-homology} of $X$ with coefficients in $A$ as the $\K$-theory group $\K_0(X;A)\coloneqq\K(\Loc(X;A))$.

\item Let $D$ be an $A$-linear Dirac operator on a complete Riemannian manifold $M$. We define its $\K$-homology class as the class $[D]\in\K_0(M;A)$ represented by the $*$-homomorphism 
\[\cS\to \Loc(M;L^2(S))\subset \Loc(M;A)\,,\quad \phi\mapsto \big(t\mapsto \phi(t^{-1}D)\big)\,.\]

\item The \emph{assembly map} $\ind\colon \K_0(X;A)\to\K(\Roe(X;A))$ is induced by evaluation at~$1$:
\[\ev_1\colon \Loc(X;A)\to\Roe(X;A)\]
\end{itemize}
Clearly, $\ind([D])=\ind(D)$.
\end{defn}

If the complete Riemannian manifold $M$ has bounded geometry (or more generally is a proper metric space with continuously bounded geometry, in the terminology of \cite[Definition 4.1]{EngelWulffZeidler}), then it was proven in \cite[Lemma 5.6]{WulffTwisted} that 
\begin{equation}\label{eq:descentasymptoticmorphism}
\Loc(M;A)\grtensor\cA(M;B)\to\fA(\Roe(M;A\grtensor B))\,,\quad L\grtensor f\mapsto [t\mapsto L_tf]
\end{equation}
defines an asymptotic morphism. It gives rise to a \emph{descent homomorphism} (cf.\ \cite[Definition 5.7]{WulffTwisted})
\[\nu_B\colon\K_0(M;A)\to\EE(\cA(M;B),\Roe(M;A\grtensor B))\]
and it is clear from the definitions that $\nu_B([D])=\llbracket D;B\rrbracket$.

\section{Operators invertible away from closed subsets}
\label{sec:invertibleawayfromclosedsubsets}

Let $X$ be a proper metric space whose Roe algebra is constructed with respect to a representation $\rho$ on a Hilbert $A$-module $\Hilbert_X$. We recall the definition of two ideals associated to a closed subspace $L\subset X$.
\begin{defn}
\begin{itemize}
\item An operator $T\in\Lin_A(\Hilbert_X)$ is \emph{supported near $L$} if there is $R>0$ such that $\rho(f)T=0=T\rho(f)$ for all $f\in\Cz(X)$ with $\dist(\supp(f),L)\geq R$.
\item Let $\Roe(L\subset X;\Hilbert_X)$ denote the norm closure of all operators in $\Roe(X;\Hilbert_X)$ that are supported near $L$. 
\item We define the ideal $\Loc(X\mid L;\Hilbert_X)\subset \Loc(X;\Hilbert_X)$ as the preimage of $\Roe(L\subset X;\Hilbert_X)$ under the evaluation at $1$ map $\ev_1$ from \Cref{sec:localizationalgebrasKhomology}.
\end{itemize}
As before we write them as $\Roe(L\subset X;A)$ and $\Loc(X\mid L;A)$ if $\Hilbert_X$ is sufficiently large.
\end{defn}

It is well-known that $\Roe(L\subset X;\Hilbert_X)$ and $\Loc(X\mid L;\Hilbert_X)$ are indeed ideals in $\Roe(X;\Hilbert_X)$ and $\Loc(X;\Hilbert_X)$, respectively.
We also recall that there are canonical natural isomorphisms 
\[\K(\Roe(L;A))\cong\K(\Roe(L\subset X;A))\,.\]
The ideal $\Loc(X\mid L;A)$ is what would be called $\Loc[L](X;A)$ in the terminology of \cite[Definition 3.1]{zeidler_secondary}. There is also another ideal called $\Loc(L\subset X;A)$ defined as the closure of all $L$ such that the support of $L_t$ becomes arbitrary close to $L$ as $t\to\infty$, but it is not relevant for our work.

We now return to the case of complete Riemannian manifolds $M$ and $A$-linear Dirac bundles $S$ with Dirac operator $D$ over it.
The following definition picks up the key idea behind \cite[Lemma 2.3]{Roe_Relative} (see also \cite[Lemma 4.4]{zeidler_secondary} and \cite[Proposition 3.15]{hanke_pape_schick}).

\begin{defn}\label{defn:invertibleawayfromL}
For all $\varepsilon>0$, let $\cS_\varepsilon\subset \cS$ denote the sub-\textCstar-algebra of all functions supported in $[-\varepsilon,\varepsilon]$.
We say that $D$ is \emph{invertible away from a closed subset $L\subset M$} if there is $\varepsilon>0$ such that
\begin{equation}\label{eq_defn_invertible_away_L}
\phi(D)\in\Roe(L\subset M;L^2(S))\subset \Roe(L\subset M;A)
\end{equation}
for all $\phi\in\cS_\varepsilon$. 
\end{defn}

\begin{example}\label{ex_invertible_upsc_outside_L}
From \cite[Lemma 2.3]{Roe_Relative} we can see that the spin Dirac operator (possibly twisted with a flat bundle) on a complete Riemannian spin manifold $M$ is invertible away from $L\subset M$ if $M$ has uniformly positive scalar curvature on the complement of a controlled neighborhood of $L$.
\end{example}

Before we define the corresponding index, let us recall the following simple facts.

\begin{lem}[{cf.\ \cite[Lemma 2.4]{zeidler_secondary}}]\label{lem:comultiplicationepsilon}
The inclusion $*$-homomorphism $\iota_\varepsilon\colon\cS_\varepsilon\hookrightarrow\cS$ 
is a homotopy equivalence of graded \textCstar-algebras.
Let $\psi_\varepsilon\colon\cS\to\cS_\varepsilon$ denote a homotopy inverse of $\iota_\varepsilon$.
The comultiplication  $\Delta\colon\cS\to\cS\grtensor\cS$ restricts to a $*$-homomorphism $\Delta_\varepsilon\colon\cS_\varepsilon\to\cS_\varepsilon\grtensor\cS_\varepsilon$ and $\Delta_\varepsilon\circ (\psi_\varepsilon\grtensor\psi_\varepsilon)$ is homotopic to $\psi_\varepsilon\circ\Delta$.\qed
\end{lem}

\begin{defn}\label{defn:invertibleawayfromLindex}
The \emph{$L$-local coarse index} and the \emph{$L$-local $\K$-homology class} of an operator $D$ that is invertible away from $L\subset M$ are defined to be the classes 
\begin{align*}
\ind(D\mid L)&\in\K(\Roe(L\subset X;A))\cong\K(\Roe(L;A))
\\ [D\mid L]&\in\K_0(M\mid L;A)\coloneqq\K(\Loc(M\mid L;A))
\end{align*}
(or in $\K(\Roe(L\subset X;L^2(S)))$ and $\K(\Loc(X\mid L;L^2(S)))$, respectively) associated to the $*$-homomorphism
\begin{align*}
\cS_\varepsilon&\to \Roe(L\subset M;L^2(S))& \phi&\mapsto \phi(D)\,,
\\\cS_\varepsilon&\to \Loc(M\mid L;L^2(S))& \phi&\mapsto \big(t\mapsto \phi(t^{-1}D)\big)\,,
\end{align*}
respectively,
composed with a homotopy inverse $\psi_\varepsilon\colon\cS\to\cS_\varepsilon$.
\end{defn}

The $L$-local index and the $L$-local $\K$-homology class are independent of the choice of the sufficiently small $\varepsilon>0$ and the homotopy inverse $\psi_\varepsilon$, because the resulting $*$-homomorphisms will be homotopic. Furthermore, we see that the canonical maps in the diagram
\begin{equation}\label{eq_diag_Lloc_classes}
\xymatrix{
\K_0(M\mid L;A)\ar[r]^-{\ind}\ar[d]&\K(\Roe(L\subset M;A))\ar[d]
\\
\K_0(M;A)\ar[r]^-{\ind}&\K(\Roe(M;A))
}
\end{equation}
map the classes $[D\mid L],[D],\ind(D\mid L)$ and $\ind(D)$ onto each other and the vertical maps are isomorphisms if $L=M$.

Unfortunately, it is not possible to define an $\EE$-theory class in $\EE(\cA(M;B), \Roe(L\subset M;S\grtensor B))$ by a formula like \eqref{eq:DiracEclassAsympMorph}, because the functions $s\mapsto\phi(t^{-1}s)$ for $\phi\in\cS$ will not lie in $\cS$ for large $t$ and hence $\phi(t^{-1}D)$ is only known to lie in the ideal $\Roe(L\subset M;A)$ for $t=1$. 
Hence we need to put up with the $\K$-homology class defined via the taylor-made ideal in the localization algebra instead.

Before working with this notion of being invertible away from a closed subset, we point out a few conditions under which it is satisfied.
The first one is a direct consequence of \cref{lem:functionalcalculusproductformula} with \cref{lem:comultiplicationepsilon}.

\begin{lem}\label{lem:ProductOpInvertible}
Let $D_1,D_2$ be $A_1$-linear resp.\ $A_2$-linear Dirac operators complete Riemannian manifolds $M_1,M_2$ and assume that they are invertible away from closed subsets $L_1\subset M_1$, $L_2\subset M_2$. Then $D_1\times D_2$ is invertible away from $L_1\times L_2\subset M_1\times M_2$.\qed
\end{lem}

\begin{lem}\label{lem:InvertibilityCheckOutsideL}
Let $M_1,M_2$ be complete Riemannian manifolds and $S_1\to M_1,S_2\to M_2$ two Dirac $A$-bundles with $A$-linear Dirac operators $D_1,D_2$.
Let furthermore $L_1\subset M_1,L_2\subset M_2$ be two closed subsets and $\psi\colon M_1\setminus L_1\to M_2\setminus L_2$ an isometry which is covered by a bundle isomorphism $\Psi\colon S_1|_{M_1\setminus L_1}\to S_2|_{M_2\setminus L_2}$ that preserves all structure of the Dirac $A$-bundles. 
Then $D_1$ is invertible away from $L_1$ iff $D_2$ is invertible away from $L_2$.
\end{lem}

\begin{proof}
For each $k=1,2$ and $r\geq 0$ let $\chi_k^r\colon M_k\to [0,1]$ be a smooth function which vanishes on the $r$-neighborhood of $L_k$ and is equal to one outside of the $(r+1)$-neighborhood. Then $\phi(D_k)\in\Roe(L_k\subset M_k;L^2(S_k))$ if and only if $\lim_{r\to\infty}\phi(D_k)\chi_k^r=0$. 
To understand this limit, we write
\[\phi(D_k)\chi_k^r= \chi_k^0\phi(D_k)\chi_k^r + (1-\chi_k^0)\phi(D_k)\chi_k^r\]
and note that the second summand converges to zero, because $\phi(D_k)$ is the norm limit of finite propagation operators.
The first summands can be compared for $k=1,2$, because they only act on the subspaces $L^2(S_k|_{M_k\setminus L_k})$ which can be identified via $\Psi$.

If $\phi\in\cS$ is a Schwarz function with Fourier transform supported in $[-R,R]$, then $\phi(D_k)$ has propagation bounded by $R$ and we immediately see that $\chi_1^0\phi(D_1)\chi_1^r=\chi_2^0\phi(D_2)\chi_2^r$ (by means of the cannonical identification) for all $r\geq R+1$. Since functions of this type are dense in $\cS$ we have
\[\lim_{r\to\infty}\chi_1^0\phi(D_1)\chi_1^r=\lim_{r\to\infty}\chi_2^0\phi(D_2)\chi_2^r\]
for all $\phi\in\cS$, in particular for all $\phi\in\cS_\varepsilon$.
The claim follows immediately.
\end{proof}

\begin{thm}[Bordism invariance]\label{thm:LocalIndexBordismInvariance}
Let $W$ be a complete Riemannian manifold with a boundary $\partial W$ which decomposes into two complete Riemanian manifolds $M_0,M_1$ and assume that there are collar neighborhoods $M_0\times[0,\varepsilon)$ of $M_0$ and $M_1\times(-\varepsilon,0]$ of $M_1$ on which the Riemannian metric of $W$ is the product metric.
By attaching infinite cylinders to it we obtain the complete Riemannian manifold $\widehat{W}\coloneqq M_0\times (-\infty,0]\cup_{M_0}W\cup_{M_1}M_1\times [0,\infty)$.
Let $D_0,D_1$ be $A$-linear Dirac operators over $M_0,M_1$ which are invertible away from closed subsets $L_1\subset M_1,L_2\subset M_2$, respectively, and assume that there is an $A\grtensor\Cl_1$-linear Dirac operator over $\widehat W$ which is equal to $D_0\times D_{(-\infty,\varepsilon)}$ and $D_1\times D_{(-\varepsilon,\infty)}$ over the respective cylinders and invertible away from a subset $\widehat{L}\coloneqq L_0\times(-\infty,0]\cup L\cup L_1\times [0,\infty)\subset \widehat W$ with $L\subset W$ containing $L_1\cup L_2$. Let $i_k\colon M_k\hookrightarrow W$ denote the canonical inclusions, restricting to the inclusions $L_k\hookrightarrow L$. Then 
\[(i_0)_*\ind(D_0\mid L_0)=(i_1)_*\ind(D_1\mid L_1)\in\K(\Roe(L;A))\,.\]
\end{thm}

\begin{proof}
We cannot preceed exactly as in the proof of \Cref{thm:BordismInvariance} (see \cite[Theorem 4.12]{WulffTwisted}), because, as noted before, we do not have the corresponding $\EE$-theory classes at our disposal. However, we can define a similar asymptotic morphism which captures the relevant information in a similar manner.

To this end, choose a smooth function $p\colon\widehat{W}\to\R$ which is equal to $(y,z)\mapsto z$ on the subsets $M_0\times(-\infty,0]$ and $M_1\times [0,\infty)$, with $p(W)\subset[-1,1]$ and whose gradient is bounded. 
For each function $f\in\Cz(\R)$ and every $t\geq 1$ we define the function $f_t\in\Cz(\R)$ by $f_t(z)\coloneqq f(t^{-1}z)$. Then $p^*f_t\in\Cb(\widehat{W})$ is given by $f_t(x)\coloneqq f(t^{-1}p(x))$ and we claim that the formula
\begin{equation}
\label{eq:LocalIndexBordismAsymptoticMorphism}
\alpha\colon \cS\grtensor\Cz(\R)\to\fA(\Roe(\widehat{W};L^2(\widehat{S})))\,,\quad \phi\grtensor f\mapsto [t\mapsto\phi(\widehat{D})p^*f_t]
\end{equation}
defines an asymptotic morphism. To verify that the universal property of the maximal tensor product indeed yields such an asymptotic morphism, we need to show that the graded commutators $[\phi(\widehat{D}),p^*f_t]$ converge in norm to zero as $t\to\infty$ and by a density argument it is sufficient to consider the generators $\phi_\pm(x)\coloneqq(x\pm i)^{-1}$ of $\cS$ and compactly supported smooth functions $f$ on $\R$. Then 
\[[\phi_{\pm}(\widehat{D}),p^*f_t]=-\phi_\pm(\widehat{D})[\widehat{D},p^*f_t]\phi_\pm(\widehat{D})\]
and the commutator $[\widehat{D},p^*f_t]$ is Clifford multiplication with the gradient of $p^*f_t$, which is of order $t^{-1}$, because the gradients of $p$ and $f$ are bounded.

Now we claim that on the sub-\textCstar-algebra $\cS\grtensor\Cz((1,\infty))$, the asymptotic morphism $\alpha$ agrees with the composition 
\begin{align*}
\alpha_1\colon \cS\grtensor\Cz(\R)\xrightarrow{\Delta\grtensor\id}\cS\grtensor\cS\grtensor\Cz(\R)&\to\fA(\Lin_{A\grtensor\Cl_1}(L^2(\widehat{S}))\\\phi_1\grtensor\phi_2\grtensor f&\mapsto[t\mapsto V_1(\phi_1(D_1)\grtensor(\phi_2(D_\R)f_t))V_1^*]
\end{align*}
where 
$V_1\colon L^2(S_1)\grtensor L^2(\R;\Cl_1)\to L^2(\widehat S)$ is any isometry which is the identity on the subspace $L^2(\widehat{S}|_{M_1\times [0,\infty})\cong L^2(S_1)\grtensor L^2([1,\infty);\Cl_1)$.
Again we check it on a dense subset, this time by assuming that $\phi\in\cS$ has Fourier transform supported in the compact interval $[-R,R]$. Then $\phi(\widehat{D})$ and $\phi(D_1\times D_\R)$ have propagation bounded by $R$ and hence they agree on sections supported in $M_1\times[R,\infty)$. But if $f\in\Cz((1,\infty))$, then $p^*f_t$ is supported in this subset for all $t\geq R$ and we immediately see that 
\[V_1\phi(\widehat{D})V_1^*(p^*f_t)=V_1\phi(D_1\times D_\R)V_1^*(p^*f_t)=V_1\phi(D_1\times D_\R)(\id\grtensor f_t)V_1^*\,.\]
The claim now follows from \cref{lem:functionalcalculusproductformula}.

Completely analogously we find an asymptotic morphism
\begin{align*}
\alpha_0\colon \cS\grtensor\Cz(\R)\xrightarrow{\Delta\grtensor\id}\cS\grtensor\cS\grtensor\Cz(\R)&\to\fA(\Lin_{A\grtensor\Cl_1}(L^2(\widehat{S}))\\\phi_1\grtensor\phi_2\grtensor f&\mapsto[t\mapsto V_0(\phi_1(D_0)\grtensor(\phi_2(D_\R)f_t))V_0^*]
\end{align*}
which agrees with $\alpha$ on $\cS\grtensor \Cz((-\infty,-1))$.

For the density arguments above, it was convenient to be able to consider special functions in $\cS$ which do not lie in $\cS_\varepsilon$. Only now that we have derived these formulas, we will restrict the asymptotic morphisms to $\cS_\varepsilon\grtensor\Cz(\R)$ for $\varepsilon>0$ small enough.

If $\phi\in\cS_\varepsilon$ and $f$ is compactly supported, then $\phi(\widehat{D})\in\Roe(\widehat{L}\subset\widehat{W};L^2(\widehat{S}))$ and $p^*f_t$ is suported in a controlled neighborhood of $W$, from which we see that 
\[\phi(\widehat{D})(p^*f_t)\in\Roe(L\subset\widehat{W};L^2(\widehat{S}))\,.\]
Furthermore, $\phi_1(D_k)\in\Roe(L_k\subset M_k;S_k)$ for $k=0,1$ and 
\[\phi_2(D_\R)f_t\in\Kom_{\Cl_1}(L^2(\R;\Cl_1))=\Kom(L^2(\R))\grtensor\Cl_1\,,
\] so that 
\[\phi_1(D_k)\grtensor (\phi_2(D_\R)f_t)\in\Roe(L_k\times\{0\}\subset M_k\times\R;S_k\grexttensprod S_\R)\,.\]
Now recall that all the solid canonical maps in the diagram 
\[\xymatrix{
\K(\Roe(L_k\times\{0\}\subset M_k\times\R;S_k\grexttensprod S_\R))\ar[d]\ar@{-->}[r]&\K(\Roe(L\subset \widehat{W};\widehat{S}))\ar[d]
\\\K(\Roe(L_k\times\{0\}\subset M_k\times\R;A\grexttensprod \Cl_1))\ar[d]\ar[d]^{\cong}&\K(\Roe(L\subset \widehat{W};A\grtensor \Cl_1)\ar[d]^{\cong}
\\\K(\Roe(L_k;A\grexttensprod \Cl_1))\ar[r]&\K(\Roe(L;A\grtensor\Cl_1))
}\]
are induced by adjoining with suitably chosen isometries. With this in mind it is readily verified that one can choose the isometries $V_0,V_1$ above in such a way that adjoining with them induces the upper dashed arrow in such a way that the diagram commutes.
In particular we then have that the images of the asymptotic morphisms $\alpha_0,\alpha_1$ restricted to $\cS_\varepsilon\grtensor\Cz(\R)$ are also contained in $\fA(\Roe(\widehat{W};L^2(\widehat{S})))$, just like the image of $\alpha$.

Let $\beta_0,\beta_1$ denote the canonical $*$-homomorphisms
\[\Cz(\R)\cong\Cz((-\infty,-1))\subset \Cz(\R)\,,\quad \Cz(\R)\cong\Cz((1,\infty))\subset \Cz(\R)\,,\]
respectively. 
Since $\beta_0,\beta_1$ are homotopic to each other and the identity, we obtain homotopies between asymptotic morphisms
\begin{align*}
\alpha_0\circ(\iota_\varepsilon\grtensor\id)&\simeq\alpha_0\circ(\iota_\varepsilon\grtensor\beta_0)=\alpha\circ(\iota_\varepsilon\grtensor\beta_0)
\\&\simeq \alpha\circ(\iota_\varepsilon\grtensor\beta_1)=\alpha_1\circ(\iota_\varepsilon\grtensor\beta_1)\simeq \alpha_1\circ(\iota_\varepsilon\grtensor\id)\,.
\end{align*}
Precomposing the left- and rightmost of these with $\psi_\varepsilon$ therefore yields the same class
\[\llbracket\alpha_0\rrbracket=\llbracket\alpha_1\rrbracket\in\EE(\Cz(\R),\Roe(L\subset\widehat{W};A\grtensor\Cl_1))\]
in $\EE$-theory.

Let $d\in\EE(\Cz(\R),\Cl_1)$ be the class of the asymptotic morphism
\[\cS\grtensor\Cz(\R)\to\fA(\Kom(L^2(\R))\grtensor\Cl_1)\,,\quad \phi\grtensor f\mapsto [t\mapsto \phi(D_\R)f_t]\]
which appears in the definition of the $\alpha_k$. 
In fact, using \Cref{lem:comultiplicationepsilon} it is straightforward to verify that $\llbracket\alpha_k\rrbracket=(\Ad_{V_k})_*\circ(\ind(D_k\mid L_k)\grtensor d)$.

Via rescaling of $\R$ it is easy to see that $d$ is also represented by the asymptotic morphism $\phi\grtensor f\mapsto [t\mapsto \phi(t^{-1}D_\R)f]$, such that it is the well-known canonical generator of $\EE(\Cz(\R),\Cl_1)$.
Furthermore, let $b\in\K(\Cz(\R)\grtensor\Cl_{0,1})$ be the Bott element and recall that $c\coloneqq (d\grtensor\id_{\Cl_{0,1}})\circ b$ is the canonical generator of $\K(\Cl_{1,1})$, see \cite{guentnerhigson}.
Then the two elements
\[(\llbracket \alpha_k\rrbracket\grtensor\id_{\Cl_{0,1}})\circ b=(\Ad_{V\grtensor\id_{\Cl_{0,1}}})_*\circ(\ind(D_k\mid L_k)\grtensor c)\]
in $\K(\Roe(L\subset \widehat{W};\widehat{S}\grtensor\Cl_{0,1}))$ agree for $k=0,1$ and they map canonically to 
\[(i_k)_*(\ind(D_k\mid L_k)\grtensor c)\in\K(\Roe(L;A\grtensor\Cl_{1,1}))\]
The claim now follows, because 
\[\xymatrix{
\K(\Roe(L_k;A))\ar[r]^{(i_k)_*}\ar[d]_{\blank\grtensor c}^{\cong}
&\K(\Roe(L;A))\ar[d]_{\blank\grtensor c}^{\cong}
\\\K(\Roe(L_k;A\grtensor\Cl_{1,1}))\ar[r]^{(i_k)_*}
&\K(\Roe(L;A\grtensor\Cl_{1,1}))
}\]
commutes and the vertical maps are isomorphisms.
\end{proof}

\section{The relative coarse index}

\subsection{Definition and the relative coarse index theorem}
\label{sec:DefRelInd}

We can now generalize the relative index defined in \cite[Section 4]{gromov_lawson_complete} to the coarse geometric set-up presented in the previous sections.

Let $M_1,M_2$ be a complete Riemannian manifold, $A$ a graded unital \textCstar-algebra and $S_1\to M_1,S_2\to M_2$ two Dirac $A$-bundles with $A$-linear Dirac operators $D_1,D_2$.
Let furthermore $W_1\subset M_1,W_2\subset M_2$ be two closed submanifolds with boundary of codimension zero and $\psi\colon M_1\setminus \mathring W_1\to M_2\setminus \mathring W_2$ an isometry which is covered by a bundle isomorphism $\Psi\colon S_1|_{M_1\setminus \mathring W_1}\to S_2|_{M_2\setminus \mathring W_2}$ that preserves all structure of the Dirac $A$-bundles.

By changing the Riemannian metrics and accordingly the Dirac bundles and Dirac operators within in a $1$-neighborhood of the boundaries $\partial W_1\subset M_1, \partial W_2\subset M_2$, we may assume that they are of product structure on a collar neighborhood $\partial W_1\times(-\varepsilon,\varepsilon),\partial W_2\times(-\varepsilon,\varepsilon)$ with $0<\varepsilon<1$ while leaving the coarse geometry of $M_1,M_2$ unchanged.
By product structure we mean that the bundles can be identified with a pullback of a bundle over $\partial W_{1,2}$ such that the grading, Clifford action and connection in $\partial W_{1,2}$-direction are constant along the interval and the connection in interval direction is simply the directional derivative.
Due to bordism invariance, this has no effect on the coarse indices that we are going to consider.

\begin{lem}
Let $S\to N\times (-\varepsilon,\varepsilon)$ be a Dirac $A$-bundle of product structure and $\tau\colon N\times (-\varepsilon,\varepsilon)\to N\times (-\varepsilon,\varepsilon)$ the diffeomorphism which reflects the interval.
The following three Dirac $A$-bundles are canonically isomorphic:
\begin{enumerate}
\item The pullback bundle $\tau^*S$ with its canonical connection and Clifford multiplication.\item The bundle $\tilde S$ which we obtain from $S$ by changing the Clifford action of vectors in the interval direction to its negative while leaving the Clifford action of vectors in $N$-direction unchanged.
\item The bundle $\bar S$ which we obtain from $S$ by reversing the grading and also changing the Clifford multiplication to its negative. 
\end{enumerate}
\end{lem}

\begin{proof}
Since $S$ has product structure, the bundle $\tau^*S$ canonically identifies with $S$, but we see that the Clifford action of vectors in interval direction will turn into its negative under this identification.

For the equivalence between $\tilde S$ and $\bar S$, let $e$ be the canonical unit vector field in positive interval direction.
Then $S\to S,\xi\mapsto e\xi$ is a grading reversing bundle isomorphism which preserves the connection and for all tangent vectors $w+\lambda e$ with $w$ in $N$-direction we calculate
\[e(w-\lambda e)\xi=-(w+\lambda e)e\xi\,,\]
so it is actually an isomorphism between $\tilde S$ and $\bar S$ which preserves all structure.
\end{proof}

Now we can glue $W_1,W_2$ together by identifying their boundaries via $\psi$ to obtain a complete Riemannian manifold $X$. 
The lemma shows that we can glue the bundle $S_1|_{W_1}$ to the bundle $S_2|_{W_2}$ with grading and Clifford action reversed and obtain a Dirac $A$-bundle $S_X\to X$ with Dirac operator $D_X$.

We arrive at a direct generalization of the relative index of Gromov and Lawson:
\begin{defn}\label{defn:PhiRelInd}
The $\Psi$-relative index of $D_1,D_2$ is 
\[
\ind(D_1,D_2\mid\Psi)\coloneqq \ind(D_X)\in\K(\Roe(X;A))\,.\qedhere
\]
\end{defn}

\begin{rem}
If $M$ is even dimensional and orientable, which is the usual case considered in index theory, then multiplication with the volume element is a grading preserving automorphism of the bundle $S_2$ which anticommutes with the Clifford action of tangent vectors. Thus, in this case it suffices to only reverse the grading on $S_2$ and not the Clifford action in order to be able to glue the bundles together.
\end{rem}

\begin{lem} \label{lem:RelIndDDidNull}
Consider the situation above with $M=M_1=M_2$, $S=S_1=S_2$, $W=W_1=W_2$, $D=D_1=D_2$ and $\psi,\Psi$ being the respective identities. Let $f\colon X\to W$ the continuous coarse map which collapses the two halves. Then
\(f_*\ind(D,D\mid \id)=0\).\end{lem}
 
\begin{proof}

We can embed $X$ into $M\times \R$ in such a way that the smaller collar neighborhood $\partial W\times (-\frac\varepsilon2,\frac\varepsilon2)\subset X$ maps canonically to $\partial W\times (-\frac\varepsilon2,\frac\varepsilon2)\subset M\times \R$ and away from the colar neighborhood $\partial W\times (-\varepsilon,\varepsilon)\subset X$ it is canonically mapped into $W\times\{+1,-1\}$. If done in the obvious way, $X$ bounds a closed submanifold $V\subset M\times\R$ of codimension zero which is coarsely equivalent to $V$. On $V$ we consider the product Dirac operator $D_V\coloneqq D\times D_\R$. After changing the Riemannian metric on $V$ near $X$ without changing its coarse structure we can assume that it has product structure near $X$. After making the corresponding modifications to the operator $D_V$, it agrees with $D_X\times D_\R$ on this product neighborhood. Hence we can apply the bordism invariance of the coarse index which yields that $\iota_*\ind(D,D\mid \id)=0$, where $\iota\colon X\hookrightarrow V$ is the inclusion constructed above. The result follows, because $f$ is close to $\iota$ composed with the canonical coarse equivalence $V\to W, (x,t)\mapsto x$.
\end{proof} 

The relative index theorem \cite[Theorem 4.35]{gromov_lawson_complete} now has the following coarse geometric generalization.

\begin{thm}\label{thm:PhiIndexTheorem}
Assume that $D_1,D_2$ are invertible away from $W_1,W_2$, respectively.
Then 
\[\ind(D_1,D_2\mid\Psi)=(\iota_1)_*\ind(D_1\mid W_1)-(\iota_2)_*\ind(D_2\mid W_2)\]
where $\iota_{1,2}\colon W_{1,2}\to X$ denote the canonical inclusions.
\end{thm}

\begin{proof}
By doing surgery on $M_1 \coprod M_2$ along the submanifold $\partial W_1\coprod\partial W_2$ we obtain a complete Riemannian manifold whose boundary consists of $M_1,M_2,X$ and a fourth part obtained by gluing $M_1\setminus\mathring W_1$ to $M_2\setminus\mathring W_2$ along $\psi|_{\partial W_1}\colon\partial W_1\to\partial W_2$. Since $\psi$ is an isometry, the fourth part is the boundary of another complete Riemannian manifold as in the proof of the previous lemma.
Thus we obtain a bordism $V$ between $M_1\coprod M_2$ and $X$ which can be written as $V=L\cup (V\setminus L)$ with isometries $V\setminus L\cong (M_1\setminus\mathring W_1)\times I\cong  (M_2\setminus\mathring W_2)\times I$ with a closed interval $I$ and inclusion $X\subset L$ being a coarsely equivalence.

\noindent
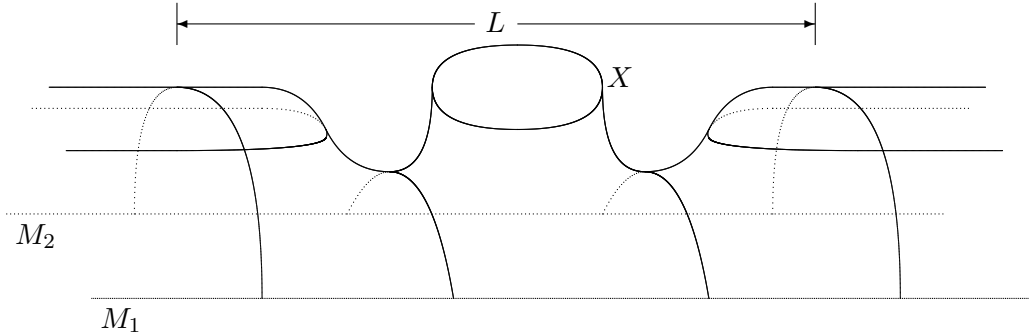
\setlength{\unitlength}{0.04\textwidth}
\begin{picture}(24,9)(-12,-4)
\qbezier(0,3)(2,3)(2,2)
\qbezier(0,3)(-2,3)(-2,2)
\qbezier(0,1)(2,1)(2,2)
\qbezier(0,1)(-2,1)(-2,2)

\qbezier(-10,-3)(1,-3)(12,-3)
\qbezier[240](-12,-1)(-1,-1)(10,-1)

\qbezier(2,2)(2,0)(3,0)
\qbezier(-2,2)(-2,0)(-3,0)

\qbezier(3,0)(4,0)(4.5,1)
\qbezier(-3,0)(-4,0)(-4.5,1)

\qbezier(4.5,1)(5,2)(6,2)
\qbezier(-4.5,1)(-5,2)(-6,2)

\qbezier(6,2)(8.5,2)(11,2)
\qbezier(-6,2)(-8.5,2)(-11,2)

\qbezier(3,0)(4,0)(4.5,-3)
\qbezier(-3,0)(-2,0)(-1.5,-3)

\qbezier[20](3,0)(2.5,0)(2,-1)
\qbezier[20](-3,0)(-3.5,0)(-4,-1)

\qbezier[20](4.5,1)(4.75,1.5)(6,1.5)
\qbezier[20](-4.5,1)(-4.75,1.5)(-6,1.5)

\qbezier[50](6,1.5)(8.3,1.5)(10.6,1.5)
\qbezier[55](-6,1.5)(-8.7,1.5)(-11.4,1.5)

\qbezier(4.5,1)(4.45,0.9)(4.5,0.8)
\qbezier(-4.5,1)(-4.45,0.9)(-4.5,0.8)

\qbezier(4.5,0.8)(4.65,0.5)(8,0.5)
\qbezier(-4.5,0.8)(-4.65,0.5)(-8,0.5)

\qbezier(8,0.5)(9.7,0.5)(11.4,0.5)
\qbezier(-8,0.5)(-9.3,0.5)(-10.6,0.5)

\qbezier(7,2)(9,2)(9,-3)
\qbezier(-8,2)(-6,2)(-6,-3)

\qbezier[45](7,2)(6,2)(6,-1)
\qbezier[45](-8,2)(-9,2)(-9,-1)

\put(-11.8,-1.7){$M_2$}
\put(-9.8,-3.7){$M_1$}
\put(2.1,2){$X$}

\put(-8,3){\line(0,1){1}}
\put(7,3){\line(0,1){1}}

\put(0,3.5){\vector(1,0){7}}
\put(-1,3.5){\vector(-1,0){7}}

\put(-0.75,3.3){$L$}
\end{picture}

Using the usual surgery construction for Dirac operators and the ideas of the previous lemma, we obtain a Dirac operator $D_V$ on $V$ which agrees with $D_1\times D_\R$ and/or $D_2\times D_\R$ on $V\setminus L$ and near $M_1,M_2$, while it agrees with $D_X\times D_\R$ near $X$. 

If we now attach infinite cylinders to the boundary of $V$ we obtain a complete Riemanian manifold $\widehat{V}$ to which $D_V$ extends canonically and we consider the closed subset $\widehat{L}$ which we obtain from $L$ by attaching infinite cylinders over $W_1,W_2$ and $X$. Then the canonical extension $D_{\widehat{V}}$ is invertible away from $\widehat{L}$ because of \Cref{lem:ProductOpInvertible,lem:InvertibilityCheckOutsideL} and the prerequisits of \Cref{thm:LocalIndexBordismInvariance} are satisfied.
It implies that 
\[i_*\ind(D_1,D_2\mid\Psi)=(i_1)_*\ind(D_1\mid W_1)-(i_2)_*\ind(D_2\mid W_2)\]
where $i\colon X\hookrightarrow L$, $i_{1,2}\colon W_{1,2}\hookrightarrow L$ are the inclusions into the three parts of the boundary of $V$.
Now note that there is a coarse equivalence $f\colon L\to X$ such that $f\circ i=\id,f\circ i_{1,2}=\iota_{1,2}$ and the claim follows.
\end{proof}

\subsection{The relative index for twisted operators}

Now we specialize to the case that we have one Dirac $A$-bundle $S\to M$ with $A$-linear Dirac operator $D$ over a complete Riemannian manifold $M$ which we twist with two different smooth $B$-bundles $E,F\to M$ that are isomorphic away from a closed subset $L\subset M$.
More precisely, we assume that there is a controlled neighborhood $U$ of $L$ and a bundle isomorphism $\Theta\colon E|_{M\setminus U}\to F|_{M\setminus U}$ and we equip the bundles with connections that are preserved under $\Theta$.

By choosing a uniformly close smooth approximation of the distance function $\dist(\blank,L)\colon M\to[0,\infty)$ and picking a sufficiently large regular value, we find a closed submanifold with boundary of codimension zero $W$ containing $U$ and with inclusion $L\hookrightarrow W$ being a coarse equivalence. Choosing $\psi$ to be the identity and $\Psi\coloneqq \id_S\grtensor\Theta|_{M\setminus \mathring W}\colon S\grtensor E|_{W\setminus\mathring W}\to S\grtensor F|_{W\setminus\mathring W} $, we have everything required for \Cref{defn:PhiRelInd}.

\begin{defn}\label{defn:TwistedRelInd}
The \emph{relative index for twisted operators} is
\[\ind(D\midd E,F)\coloneqq f_*\ind(D_E,D_F\mid\id_S\grtensor \Theta|_{M\setminus \mathring W})\in\K(\Roe(L;A\grtensor B))\]
where $f\colon X\to L$ identifies the two copies of $W$ in $X$ followed by a coarse inverse of the inclusion $L\hookrightarrow W$.
\end{defn}

The isomorphism $\Theta$ is understood implicitly and therefore we omitted it from the notation. In the following, it will usually be the identity anyway.
It follows readily from bordism invariance that this index is independent of the choice of connections on $E,F$ and the submanifold $W$.

Our focus in the rest of the paper will be on even more special $B$-bundles, for which the relative index for twisted operators has a very convenient alternative description. We now introduce them along with some related terminology.
\begin{defn}\label{defn_A0_subsubset}
\begin{itemize}
\item We say that $E,F\to M$ are \emph{$B$-bundles of bounded variation which agree away from $L$} if their defining projection valued functions $P,Q\in\cA(M;B)$ are equal on the complement of some controlled neighborhood $U$ of~$L$.

\item Let $\cAz(L\Subset M;B)\subset \cA(M;B)$ denote the closure of all functions whose support lies in some controlled neighborhood of $L$, i.\,e. it is the closure of the union of all $\cA(\overline{U},\partial U;B)$ where $U$ runs over all controlled neighborhoods of $L$.

\item Since $P-Q\in\cA(\overline{U},\partial U;B)\subset \cAz(L\Subset M;B)$, these projections give rise to a $\K$-theory class $\llbracket E\rrbracket-\llbracket F\rrbracket\in\K(\cAz(L\Subset M;B))$.
\qedhere
\end{itemize}
\end{defn}

We use the symbol $\Subset$ above to indicate that the inclusion is only supposed to be understood in a coarse geometric sense, because otherwise one might misinterpret it as an algebra of functions supported in $L$, which it is not.
It is an ideal which generalizes $\Cz$-functions. Indeed, if $L$ is compact and intersects all connected components of $M$, then $\cAz(L\Subset M;B)=\Cz(M)\grtensor B\grtensor\grKom$.

Clearly, \eqref{eq:DiracEclassAsympMorph} restricts to an asymptotic morphism
\[
\cS\grtensor \cAz(L\Subset M;B)\to\fA( \Roe(L\subset M;S\grtensor B))\,,\]
because if $f\in\cAz(L\Subset M;B)$ is contained in an $R$-neighborhood of $L$ and $\phi\in\cS$ has Fourier transform supported in $[-S,S]$, then $\phi(t^{-1}D)$ has propagation bounded by $t^{-1}S$ and $(\phi(t^{-1}D)\grtensor \id_{B\grtensor\grKom}) \cdot (\id_S\grtensor f)$ is supported within the $(R+t^{-1}S)$-neighborhood of $L$.

\begin{defn} 
The $\EE$-theory class of $D$ localized near $L$ is the class
\[\llbracket D\midd L; B\rrbracket \in\EE(\cAz(L\Subset M;B),\Roe(L\subset M;A\grtensor B))\]
represented by \eqref{eq:DiracEclassAsympMorph} with restricted domain and target.
\end{defn}

\begin{rem}\label{rem:notationE}
Now that we have introduced several very similar and perhaps confusing notations, we should probably provide a mnemonic for how we use them: In short, we use a single bar if something is already ``$L$-local'' and a double bar if we perform a procedure to ``localize'' it near $L$. 
So in contrast to the $L$-local index $\ind(D\mid L)$ and the $L$-local $\K$-homology class $[D\mid L]$, where the locality was already inherent to $D$, here we ``localize'' near $L$ to obtain the class $\llbracket D\midd L; B\rrbracket$ (and similarly the index $\ind(D\midd E,F)$) by restricting to the algebra $\cAz(L\Subset M;B)$ of functions which are already $L$-local. 
\end{rem}

We now have the following alternative description of the relative index for twisted operators, where we have made the canonical identification $\K(\Roe(L;A\grtensor B))\cong\K(\Roe(L\subset M;A\grtensor B))$.

\begin{thm}\label{thm:TwistedRelIndETheory}
\(\ind(D\midd E,F)=\llbracket D\midd L;B\rrbracket\circ (\llbracket E\rrbracket-\llbracket F\rrbracket)\)
\end{thm}

\begin{proof}
Let $E'$ denote the bundle obtained by gluing together a copy of $E|_W$ and a copy of $F|_W$ along the boundary and $F'$ the bundle obtained by gluing two copies of $F$ together. Furthermore, let $D'$ denote the operator on $X$ obtained by gluing together two copies of the Dirac bundle underlying $D$. Then the indices of the twisted operators $D'_{E'}$, $D'_{F'}$ are by definition of the relative indices and \Cref{lem:RelIndDDidNull}
\[f_*\ind(D'_{E'})=\ind(D\midd E,F)\,,\quad f_*\ind(D'_{F'})=\ind(D\midd F,F)=0\,.\]
On the other hand we have $\llbracket E\rrbracket-\llbracket F\rrbracket=\llbracket E'\rrbracket-\llbracket F'\rrbracket\in\K(\cA( W,\partial W; B))$ and hence
\begin{align*}
f_*(\ind(D'_{E'})-\ind(D'_{F'}))&=f_*(\llbracket D';B\rrbracket\circ (\llbracket E'\rrbracket-\llbracket F'\rrbracket))
\\&=g_*\llbracket D'|_{\mathring W};B\rrbracket\circ (\llbracket E'\rrbracket-\llbracket F'\rrbracket)
\\&=g_*\llbracket D|_{\mathring W};B\rrbracket\circ (\llbracket E\rrbracket-\llbracket F\rrbracket)
\\&=\llbracket D\midd L;B\rrbracket\circ (\llbracket E\rrbracket-\llbracket F\rrbracket)
\end{align*}
where $g\colon W\to L$ denotes a coarse inverse of the inclusion, the third equality is a consequence of \Cref{lem:independence} and the last equality is clear from the definitions.
\end{proof}

\subsection{Twisting at the level of \texorpdfstring{$\K$}{K}-homology classes}
\label{sec:twistingKtheory}

If $M$ has bounded geometry, then the asymptotic morphism \eqref{eq:descentasymptoticmorphism} clearly restricts to 
\begin{equation}\label{eq:restricteddescentasymptoticmorphism}
\Loc(M;A)\grtensor\cA(L\Subset M;B)\to\fA(\Roe(L\subset M;A\grtensor B))
\end{equation}
thereby giving rise to a descent homomorphism
\[\nu_{B,L}\colon\K_0(M;A)\to\EE(\cA(L\Subset M;B),\Roe(L\subset M;A\grtensor B))\]
and again it is clear that $\nu_{B,L}([D])=\llbracket D\midd L;B\rrbracket$.
Combining this with \Cref{thm:PhiIndexTheorem} and \Cref{thm:TwistedRelIndETheory} 
we see that if $D_E,D_F$ are invertible away from $L$ and $E,F\to M$ agree away from $L$, then 
\begin{align*}
\ind(D_E\mid L)-\ind(D_F\mid L)&=\ind(D\midd E,F)
\\&=\llbracket D\midd L;B\rrbracket\circ (\llbracket E\rrbracket-\llbracket F\rrbracket)
\\&=\nu_{B,L}([D])\circ (\llbracket E\rrbracket-\llbracket F\rrbracket)
\end{align*}  
This raises the question whether there exists the dashed homomorphism $\tw$ making the following diagram commute:
\begin{equation}\label{eq:Khomologytwistingdiagram}
\xymatrix{
\K_0(M;A)\otimes\K(\cAz(L\Subset M;B))\ar[d]^-{\nu_{B,L}\otimes \id}\ar@{-->}[r]^-{\tw}
&\K_0(M\mid L;A\grtensor B)\ar[d]_-{\ind}
\\{\begin{matrix}\EE(\cAz(L\Subset M;B),\Roe(L\subset M;A\grtensor B))\\\otimes\K(\cAz(L\Subset M;B))\end{matrix}}\ar[r]^-{\circ}
&\K(\Roe(L\subset M;A\grtensor B))
}\end{equation}  
The map $\tw$ is supposed to 
satisfy 
\begin{equation}\label{eq:Khomologytwistingformula}
\tw([D]\grtensor (\llbracket E\rrbracket-\llbracket F\rrbracket))=[D_E\mid L]- [D_F\mid L]
\end{equation}
and, following Atiyah's idea from \cite{atiyah_global_theory_elliptic_operators} that elements of $\K$-homology should be seen as generalized elliptic operators, it should therefore be understood as ``twisting the generalized operator $[D]$ by the virtual bundle $\llbracket E\rrbracket-\llbracket F\rrbracket$ which vanishes away from $L$ to obtain the generalized operator $[D_E\mid L]- [D_F\mid L]$ that is invertible away from $L$''.
Note that this is something which does not work at the level of differential operators itself. 
It should also be remarked that the twisting map $\tw$ could also be interpreted (and written) as a cap product, but we will not do it here.

Defining the map is straightforward. Note that the image of asymptotic morphism \eqref{eq:restricteddescentasymptoticmorphism} is contained in $\Loc(M\mid L;A\grtensor B)/\Cz([1,\infty);\Roe(L\subset M;A\grtensor B))$ and hence it induces a map
\begin{align*}
\tw\colon \K_0(M;A)\otimes\K(\cAz(L\Subset M;B))&\to\K\left(\frac{\Loc(M\mid L;A\grtensor B)}{\Cz([1,\infty);\Roe(L\subset M;A\grtensor B))}\right)
\\&\cong \K(\Loc(M\mid L;A\grtensor B))\cong\K_0(M\mid L;A\grtensor B)\,.
\end{align*}
Moreover, it is readily verified that it makes the diagram \eqref{eq:Khomologytwistingdiagram} commute.
The hard part is to prove that formula \eqref{eq:Khomologytwistingformula} holds.

We start by noting that the following lemma can be shown by essentially the same calculations as those in the proof of \cite[Theorem 4.14]{WulffTwisted}.
\begin{lem}\label{lem:TwistedKHomologyClassRepresentative}
Under the isomorphisms above, the $\K$-homology class $[D_E\mid L]$ is represented by the $*$-homomorphism
\begin{align*}
\beta_P\colon \cS_\varepsilon&\to \frac{\Loc(M\mid L;A\grtensor B)}{\Cz([1,\infty);\Roe(L\subset M;A\grtensor B))}
\\\phi&\mapsto \left(t\mapsto (V^*\phi(t^{-1}D) V)\grtensor\id_{B\grtensor\grelltwo})\circ(\id\grtensor P) \right)
\end{align*}
up to rearrangement of the tensor factors. Analogously we obtain a $*$-homomorphism $\beta_Q$ representing $[D_F\mid L]$ for $F,Q$ instead of $E,P$.\qed
\end{lem}

At this point we have to recall how exactly two $*$-homomorphisms $\alpha^+,\alpha^-\colon\cS\to D$ into a \textCstar-algebra $D$ which agree modulo an ideal $J$ determine a difference element in $\K(J)$. 
Applied to the two $*$-homomorphisms $\alpha_P,\alpha_Q\colon\cS\to \cA(M;B)$ given by $\alpha_P(\phi)\coloneqq\phi(0)\cdot P$ and $\alpha_Q(\phi)\coloneqq\phi(0)\cdot Q$ this will be the construction that determines $\llbracket E\rrbracket-\llbracket F\rrbracket\in\K(\cA(L\Subset M;B))$.
It can also be applied to the $*$-homomorphisms $\beta_P,\beta_Q$ from the preceding lemma in slightly greater generality: Even if $D_E,D_F$ are not invertible away from $L$ and hence their images are only contained in $\frac{\Loc(M;A\grtensor B)}{\Cz([1,\infty);\Roe(L\subset M;A\grtensor B))}$, they clearly agree modulo the ideal $\frac{\Loc(M\mid L;A\grtensor B)}{\Cz([1,\infty);\Roe(L\subset M;A\grtensor B))}$ and we can use the construction to define a difference element $[D_E]-[D_F]\in\K_0(M\mid L;A\grtensor B)$. 

Let $M_{1,1}(\C)$ denote the \textCstar-algebra of $2\times 2$-matrices with the diagonal-off-diagonal grading and consider the mapping cone \textCstar-algebra
\[C=C(\pi)\coloneqq\left\{(c,f)\in D\oplus \Cz\left([0,\infty);D/J)\right)\,\middle|\, \pi(c)=f(0) \right\}\]
of the quotient map $\pi\colon D\to D/C$.
Then $J$ embeds canonically as the ideal $\{(j,0)\mid j\in J\}$ into $C$ and the quotient $C/J\cong\Cz([0,\infty);D/J) $ is contractible, so we have an isomorphism
\[\K(J)\cong \K(C)\cong\K(C\grtensor M_{1,1}(\C))\]
where, to be specific, the second isomorphism is obtained by the map $C\to C\grtensor M_{1,1}(\C)$ which tensors with $(\begin{smallmatrix}1&0\\0&0\end{smallmatrix})$. (Tensoring with $(\begin{smallmatrix}0&0\\0&1\end{smallmatrix})$ instead would yield the negative of this isomorphism.)

For $s\in[1,\infty)$ the $*$-homomorphisms
\[\psi_s\colon \cS\to M_{1,1}(\C)\,,\quad \phi\mapsto \phi\begin{pmatrix}s&0\\0&s\end{pmatrix}=\begin{pmatrix}\phi_0(s)&\phi_1(s)\\\phi_1(s)&\phi_0(s)\end{pmatrix}\,,\]
where $\phi_0,\phi_1$ denote the even and odd parts of $\phi$, respectively, constitute a homotopy between $\psi_0\colon\phi\mapsto\phi(0)\cdot(\begin{smallmatrix}1&0\\0&1\end{smallmatrix})$ and $\psi_\infty=0$. Let $\bar\alpha\coloneqq\pi\circ\alpha^+=\pi\circ\alpha^-\colon \cS\to D/J$ be the quotient $*$-homomorphism of $\alpha^+$ and $\alpha^-$. Then the $*$-homomorphisms $(\bar\alpha\grtensor\psi_s)\circ\Delta\colon\cS\to D/J\grtensor M_{1,1}$ constitute a homotopy between $(\bar\alpha\grtensor\psi_0)\circ\Delta=\bar\alpha\grtensor (\begin{smallmatrix}1&0\\0&1\end{smallmatrix})$ and the zero $*$-homomorphism. Combining this homotopy with the $*$-homomorphism $(\begin{smallmatrix}\alpha^+&0\\0&\alpha^-\end{smallmatrix})\colon\cS\to D\grtensor M_{1,1}$ now yields an $*$-homomophism $\cS\to C_1\grtensor M_{1,1}$ which gives us the element of $\K(J)$.
It is easy to check that if $\alpha^+,\alpha^-$ are $*$-homomorphisms with image in the ideal $J$ itself, then this construction yields $[\alpha^+]-[\alpha^-]\in\K(J)$ and hence is indeed the one we want.

\begin{thm}
With the difference elements $\llbracket E\rrbracket-\llbracket F\rrbracket\in\K(\cA(L\Subset M;B))$ and $[D_E]-[D_F]\in\K_0(M\mid L;A\grtensor B)$ as above, the formula
\[\tw([D]\grtensor (\llbracket E\rrbracket-\llbracket F\rrbracket))=[D_E]- [D_F]\]
holds.
\end{thm}

\begin{proof}
Let $C_1$ be the mapping cone of the quotient map
\[
\pi_1\colon \cA(M;B)\to\cA(M;B)/\cA(L\Subset M;B)
\]
and $C_2$ the mapping cone of the quotient map 
\[\pi_2\colon\frac{\Loc(M;A\grtensor B)}{\Cz([1,\infty);\Roe(L\subset M;A\grtensor B))}\to\frac{\Loc(M;A\grtensor B)}{\Loc(M\mid L;A\grtensor B)}\,.\]
Then the $*$-homomorphism 
\[\Loc(M;A)\grtensor\cA(L\Subset M;B)\to\frac{\Loc(M\mid L;A\grtensor B)}{\Cz([1,\infty);\Roe(L\subset M;A\grtensor B))}\]
inducing the twist map $\tw$ (cf.\ \eqref{eq:restricteddescentasymptoticmorphism}) is readily seen to extend to a $*$-homomorphism \(\Loc(M;A)\grtensor C_1\grtensor M_{1,1}\to C_2\grtensor M_{1,1}\) and so we obtain a commutative diagram
\[\xymatrix{
\K_0(M;A)\otimes\K(\cAz(L\Subset M;B))\ar[r]^-{\tw}\ar[d]^{\cong}&\K_0(M\mid L;A\grtensor B)\ar[d]^{\cong}
\\\K_0(M;A)\otimes\K(C_1\grtensor M_{1,1}(\C))\ar[r]&\K(C_2\grtensor M_{1,1}(\C))\,.
}\]
It therefore suffices to check that the representatives for the classes in the lower row obtained are mapped to each other, but this is obvious from the formulas for $\beta_P,\beta_Q$ in \Cref{lem:TwistedKHomologyClassRepresentative}.
\end{proof}

\section{Relation to the Thom class}
\label{sec:Thomclassrelation}

Let $M$, as always, be a complete Riemannian manifold and let $N\subset M$ be a complete submanifold of codimension $r$ with $\K$-orientable normal bundle $\pi\colon V\to N$.  
We are going to consider a Dirac $A\grtensor\Cl_r$-bundle $S_M\to M$ with Dirac operator $D_M$ and a Dirac $A$-bundle $S_N\to N$ with Dirac operator $D_N$.
We assume that they are related as follows.
Fix a $\K$-orientation on $V$, i.\,e.\ a principle $\SpinC_r$-bundle $P_{\SpinC_r}(V)\to N$ and a principle $\unitary_1$-bundle $P_{\unitary_1}(V)$ together with a $\SpinC_r$-equivariant bundle map $P_{\SpinC_r}(V)\to P_{\SO_r}(V)\times P_{\unitary_1}(V)$, where $P_{\SO_r}(V)$ denotes the principle $\SO_r$-bundle of $V$.
Let $\ell$ denote the left multiplication of $\SpinC_r\subset \Cl_r$ on $\Cl_r$ and 
\[S_V\coloneqq P_{\SpinC_r}(V)\times_\ell\Cl_r\]
the associated spinor $\Cl_r$-bundle.
The normal bundle $V$ inherits a bundle metric and metric connection from $TM|_N$.
It induces a connection on $\Cl(V)$ and together with the choice of a connection on $P_{\unitary_1}(V)$ it also gives rise to a $\Cl_r$-linear connection on $S_V$ which is compatible with the action of $\Cl(V)$.
The relation between $S_M$ and $S_N$ that we require is that the bundles $S_M|_N$ and $S_N\grtensor S_V$ are isomorphic via an isomorphism that preserves the $\Cl(TM)|_N\cong\Cl(TN)\grtensor\Cl(V)$-actions.

Now, let $U\subset M$ be a controlled tubular neighborhood of $N$ and consider two $\Cl_{0,r}$-bundles $E,F\to M$ which agree outside of $U$ and such that their difference $\tau\coloneqq \llbracket E\rrbracket-\llbracket F\rrbracket$ represents a Thom class. We leave it open for now to specify the $\K$-theory group in which $\tau$ lives.
Independent of that, being a Thom class means that for every compact subset $K\subset N$ such that $V$ is trivial over $N$, the restriction of $\tau$ to $U\cap \pi^{-1}(K)$ is a generator of the free $\K(K)$-module $\K^r(U\cap \pi^{-1}(K))$, cf \cite[Definition C.6 with Theorem C.7]{lawson_michelsohn}.
The goal of this section is to show that the relative index with respect to these bundles is 
\begin{equation}\label{eq_computation_Thom_one}
\ind(D_M\midd E,F)=i_*\ind(D_N)\,,
\end{equation}
where $i\colon N\hookrightarrow M$ denotes the inclusion inducing a canonical homomorphism
\begin{equation}\label{eq_inclusion_N_to_M}
i_*\colon\K(\Roe(N;A))\to \K(\Roe(N\subset M;A)) \cong \K(\Roe(N\subset M; A\grtensor\Cl_r\grtensor\Cl_{0,r}))\,.
\end{equation}

We will approach this task via our formulas in $\EE$-theory. First we note that if $E,F$ are bundles of bounded variation so that $\tau$ is an element of $\K(\cAz(N\Subset M;\Cl_{0,r}))$, then the equality above is equivalent to
\begin{equation}\label{eq_computation_Thom_two}
\llbracket D_M\midd N;\Cl_{0,r}\rrbracket\circ \tau=i_*\ind(D_N)
\end{equation}
by \Cref{thm:TwistedRelIndETheory}.
The advantage of \eqref{eq_computation_Thom_two} over \eqref{eq_computation_Thom_one} is that it does not mention the bundles $E$ and $F$ anymore.
There is always a nice description of the Thom class $\tau$ in all dimensions (see \eqref{eq_Thom_one} and \eqref{eq_Thom_two} below), which can be used for the calculations, whereas the bundles $E,F$ themselves are only inherent to the even dimensions.

\subsection{Proof under additional assumptions}

In this section we are going to prove \eqref{eq_computation_Thom_two} under additional assumptions, which facilitate the computations. In the subsequent section we are going to discuss how to relax them.
The assumptions are that the bundle $S_M$ and hence also the Dirac operator $D_M$ take on a particularly nice form on a controlled tubular neighborhood of $N$.

First of all, the total space of the normal bundle $V$ carries a canonical Riemannian metric such that the canonical isomorphisms $T_wV\to T_{\pi(w)}N\oplus V_{\pi(w)}$ given by the connection are isometric.

\begin{assumption}\label{assump_thom_one}
We assume that $U$ is isometric to the $\varepsilon$-disc bundle in $V$ for some $\varepsilon>0$.
\end{assumption}

Second, note that the pullback bundle $\pi^*(S_N\grtensor S_V)$ is canonically a Dirac $A\grtensor \Cl_r$-bundle over $V$. 

\begin{assumption}\label{assump_thom_two}
We assume that $S_M$ agrees with $\pi^*(S_N\grtensor S_V)$ on the tubular neighborhood $U$.
\end{assumption}

The third assumption that we are going to make will enable us to define a suitable Thom class.
Note that the two spin groups $\Spin_r\subset\ClR_r$ and $\Spin_{0,r}\subset\ClR_{0,r}$ are canonically isomorphic by mapping $v_1\cdots v_{2s}\to (-1)^sv_1\cdots v_{2s}$ and hence the same is true for the groups $\SpinC_r\subset\Cl_r$ and $\SpinC_{0,r}\subset\Cl_{0,r}$. We can therefore consider  $P_{\SpinC_r}(V)$ also as a principle $\SpinC_{0,r}$-bundle and via the left multiplication $\ell'$ of $\SpinC_{0,r}$ on $\Cl_{0,r}$ we obtain the associated $\Cl_{0,r}$-bundle 
\[S'_V\coloneqq P_{\SpinC_r}(V)\times_{\ell'}\Cl_{0,r}\,.\]
The bundle $\Cl(-V)$ whose fiber at $x\in N$ is the Clifford algebra of $V_x$ formed with the negative scalar product is isomorphic to the associated bundle
\( P_{\SpinC_r}(V)\times_{\Ad}\Cl_{0,r} \) with respect to the adjoint action $\Ad$ and hence $S'_V$ is a bundle of left modules over the bundle of algebras $\Cl(-V)$, cf.\ \cite[Proposition 3.8]{lawson_michelsohn}. In fact, $\Cl(-V)$ is equal to the bundle $\Lin_{\Cl_{0,r}}(S'_V)$ of $\Cl_{0,r}$-linear operators on the fibers of $S'_V$.

The inclusion of $V$ into $\Cl(-V)$ gives rise to a canonical section $C\in \Gamma(\pi^*\Cl(-V))$ and we thus obtain a $*$-homomorphism 
\begin{equation}\label{eq_Thom_one}
\beta\colon \cS\to  \Cb(V;\pi^*\Cl(-V))=\Cb(V;\pi^*\Lin_{\Cl_{0,r}}(S'_V))\,,\quad \varphi\mapsto \varphi(C)
\end{equation}
into the \textCstar-algebra of bounded continuous sections of the bundle $\pi^*\Cl(-V)=\pi^*\Lin_{\Cl_{0,r}}(S'_V)\to V$.
Over each fiber $V_x$ it restricts to the Bott element (see \cite[Definition 1.26]{guentnerhigson}) and therefore it is clear that it will represent a Thom class in any meaningful $\K$-theory group.

\begin{assumption}\label{assump_thom_three}
We assume that $S'_V$ can be embedded isometrically as a direct summand of bounded variation into a trivial bundle $(\Cl_{0,r}\grtensor\grelltwo)\times N\to N$ such that under the induced inclusion $\Cb(V;\pi^*\Lin_{\Cl_{0,r}}(S'_V))\subset \Cb(V;\Cl_{0,r}\grtensor\grKom)$ the image of $\beta$ will be contained in $\cAz(N\Subset V;\Cl_{0,r})$.
\end{assumption}

Consider as in \Cref{sec:invertibleawayfromclosedsubsets} a homotopy inverse $\psi_\varepsilon$ of the inclusion $\cS_\varepsilon\subset\cS$.
Since we assumed that $U$ is the $\varepsilon$-disk bundle of $V$, we have $\varphi(C)\in\cA(\overline U,\partial U;\Cl_{0,r})$ for all $\varphi\in\cS_\varepsilon$ and $\beta$ is clearly homotopic to
\begin{equation}\label{eq_Thom_two}
\beta'\coloneqq\beta\circ\psi_\varepsilon \colon\cS\to \cA(\overline U,\partial U;\Cl_{0,r})
\end{equation}
followed up with the inclusion $\cA(\overline U,\partial U;\Cl_{0,r})\subset\cA(N\Subset V;\Cl_{0,r})$.

\begin{defn}\label{def:Thomclassassumptions}
The Thom class $\tau$ for our set-up is the image of $\tau'\coloneqq [\beta']$ under the canonical map $\K(\cA(\overline U,\partial U;\Cl_{0,r}))\to \K(\cAz(N\Subset M;\Cl_{0,r}))$ induced by the inclusion $*$-homomorphism.
\end{defn}

Note that this Thom class is independent of the sufficiently small $\varepsilon>0$, i.\,e.\ the size of the tubular neighborhood $U$, because the $\psi_\varepsilon\colon\cS\to\cS_\varepsilon\subset\cS$ are all homotopic to each other.

Now we can start the proof of \eqref{eq_computation_Thom_two} under the assumptions above by observing some general facts.
\begin{lem}\label{lem:parallelorthogonal}
Assume \Cref{assump_thom_one,assump_thom_two}.
Let $(e_1,\dots,e_n,f_1,\dots,f_r)$ be a local orthonormal frame of $TU\subset TV$ which is synchronous at a point $w\in U$ and such that $e_1,\dots,e_n\in TN, f_1,\dots,f_r\in V$ in the canonical decomposition. 
\begin{enumerate}
\item At the point $w$, the Levi-Civita connection $\nabla^M$ of $M$ satisfies $[\nabla^M_{e_k},\nabla^M_{f_l}]=0$.
\item The same is true for the connection $\nabla^{S_M}$ of the bundle $S_M$.
\item The Dirac operator $D_M$ of $S_M$ decomposes over $U$ as $D_M=D_\| + D_\perp$ such that locally 
\[D_\|=\sum_{k=1}^n e_k\nabla^{\pi^*(S_N\grtensor S_V)}_{e_k}\,,\quad D_\perp=\sum_{l=1}^r f_l\nabla^{\pi^*(S_N\grtensor S_V)}_{f_l}\,.\]
Note that $D_\|$ and $D_\perp$ can actually be defined on the whole bundle $\pi^*(S_N\grtensor S_V)$ over $V$ and not only over $U$.

\item Let $\xi$ and $\zeta$ be local sections of $S_N$ and $\pi^*(S_V)$, respectively. Then
\[
D_\perp(\pi^*\xi\grtensor\zeta)=\pi^*\xi\grtensor D_V\zeta\,,
\]
where the operator $D_V=\sum_{l=1}^r f_l\partial_{f_l}$ is the family of Dirac operators on the fibers of $V$.
\item\label{enum:parallelorthogonalhorizontalpart} If in addition $\zeta$ is parallel in the horizontal direction (i.\,e.\ $\nabla^{\pi^*S_V}_{e_k}\zeta=0$ for all $k=1,\dots,n$), then $D_\|(\pi^*\xi\grtensor\zeta)=\pi^*(D_N\xi)\grtensor\zeta$.
\item The graded commutator $[D_\|,D_\perp]=D_\|D_\perp+D_\perp D_\|$ vanishes. Hence for all functions $\varphi,\psi\in \cS$, the graded commutator $[\varphi(D_\|),\psi(D_\perp)]$ vanishes.

\item\label{enum:parallelorthogonaluvformula} Consider the generators $u(t)\coloneqq e^{-t^2}$ and $v(t)\coloneqq te^{-t^2}$ of $\cS$. Then for all $r,s>0$:
\begin{align*}
u(rD_\|+sD_\perp)&=u(rD_\|)u(sD_\perp)\\v(rD_\|+sD_\perp)&=u(rD_\|)v(sD_\perp)+v(rD_\|)u(sD_\perp)
\end{align*}
As a consequence, the $*$-homomorphism $\varphi\mapsto \varphi(rD_\|+sD_\perp)$ is equal to the composition of the comultiplication $\Delta\colon \cS\to\cS\grtensor\cS$ with $\varphi\grtensor\psi\mapsto\varphi(rD_\|)\psi(sD_\perp)$.

\item\label{enum:parallelorthogonalseparatecontinuity} For all $\varphi\in\cS$, the function
\[(0,\infty)^2\to \Roe(V;A\grtensor\Cl_r)\,,\quad (r,s)\mapsto \varphi(rD_\|+sD_\perp) \]
is continuous.
\end{enumerate}
\end{lem}

\begin{proof}
Due to the choice of the Riemannian metric on $V$, the flow of the vector field $e_k$ induces an isometric embedding $(-\varepsilon,\varepsilon)\times \R^r\cong (-\varepsilon,\varepsilon)\times V_{\pi(v)}\hookrightarrow V$ under which the derivatives $\nabla^M_{e_k},\nabla^M_{f_l}$ ($k$ fixed, $l=1,\dots,r$) correspond to partial derivatives in $\R^{r+1}$ and hence commute with each other. This shows the first claim.

The second claim follows readily from the first one, because the connection of the bundle $\pi^*(S_N\grtensor S_V)$ is derived from the Levi-Civita connection on $M$.

The third to fifth claim is obvious, considering that $\nabla^{\pi^*(S_N\grtensor S_V)}_{f_l}$ acts over a fiber of $V$ simply as the directional derivative $\partial_{f_l}$ on a trivial bundle.

The first part of the sixth claim follows at each arbitrarily chosen point $w$ from the second claim and the fact, that the Clifford action of the $e_k$ anticommutes with the Clifford action of the $f_l$. The second part is a direct consequence of a well-known property of the functional calculus, since $D_\|$ graded commutes with $D_\perp$.

The first part of the seventh claim is proven just like in \cite[Appendix A.4]{higsonkasparovtrout}: Let $\varphi\in\cS$ be a compactly supported even function. Then the operators $D_\|$ and $D_\perp$ are bounded on the subspace 
\[H_\varphi\coloneqq \varphi(D_\|)\varphi(D_\perp)L^2(\pi^*(S_N\grtensor S_V))=\varphi(D_\perp)\varphi(D_\|)L^2(\pi^*(S_N\grtensor S_V))\]
 and one can verify the equalities on $H_\varphi$ by approximating the operators by polynomials in $D_\|$ and $D_\perp$. Since the union of the $H_\varphi$ is dense in $L^2(S_M)$, they are true in general.
The second part follows because of the well-known formulas $\Delta u=u\grtensor u$ and $\Delta v=u\grtensor v+v\grtensor u$.
 
The last claim can be checked on the generators $u,v$ and then the continuity is obvious from the formulas of the seventh claim and norm continuity of functional calculus. Alternatively, one can verify it on the functions $\varphi_\pm(t)\coloneqq (t\pm i)^{-1}$, which also generate $\cS$. Here, we first note that we have positive operators
\begin{align*}
0\leq &(rD_\|\varphi_\pm(rD_\|+sD_\perp))^*(rD_\|\varphi_\pm(rD_\|+sD_\perp))
\\&=\varphi_\pm(rD_\|+sD_\perp)^*r^2D_\|^2 \varphi_\pm(rD_\|+sD_\perp)
\\&\leq \varphi_\pm(rD_\|+sD_\perp)^*(r^2D_\|^2+s^2D_\perp^2) \varphi_\pm(rD_\|+sD_\perp)
\\&= \varphi_\pm(rD_\|+sD_\perp)^*(rD_\|+sD_\perp)^2 \varphi_\pm(rD_\|+sD_\perp)
\\&=\psi(rD_\|+sD_\perp)
\end{align*}
with $\psi(t)=\frac{t^2}{1+t^2}$ and hence $\|rD_\|\varphi_\pm(rD_\|+sD_\perp)\|\leq 1$. Then for all $r,r'>0$ and $s>0$ we have
\begin{align*}
\|\varphi_\pm&(r'D_\|+sD_\perp)-\varphi_\pm(rD_\|+sD_\perp)\|=
\\&=\|\varphi_\pm(r'D_\|+sD_\perp)(rD_\|-r'D_\|)\varphi_\pm(rD_\|+sD_\perp)\|
\\&\leq \|\varphi_\pm(r'D_\|+sD_\perp)\| \cdot \frac{|r-r'|}{r}\cdot  \|rD_\|\varphi_\pm(r'D_\|+sD_\perp)\|
\\&\leq \frac{|r-r'|}{r}\,,
\end{align*}
showing equicontinuity in $r$ when seen as a family of functions parametrized in $s$
and we obtain a similar formula for equicontinuity in $s$.
\end{proof}

\begin{thm}\label{thm:Thomclassrelation}
Under the \cref{assump_thom_one,assump_thom_two,assump_thom_three} 
we get
\begin{equation}\label{eq_thom_main}
\llbracket D_M\midd N;\Cl_{0,r}\rrbracket\circ \tau=i_*\ind(D_N)\,.
\end{equation}
\end{thm}

\begin{proof}
Since the path metric of $V$ restricted to $U$ is clearly greater or equal to the path metric of $M$ restricted to $U$ and the $1$-Lipschitz identity map $\id\colon \overline{U}^{V}\to\overline{U}^M$ canonically identifies with the inclusion $i\colon N\hookrightarrow M$ up to coarse equivalence, we can apply \Cref{lem:independence} and see that 
\[\llbracket D_M\midd N;\Cl_{0,r}\rrbracket\circ \tau=\llbracket D_M|_U;\Cl_{0,r}\rrbracket\circ \tau'=i_*\circ \llbracket \hat D|_U;\Cl_{0,r}\rrbracket\circ \tau'\]
where $\hat D$ denotes the Dirac operator of the Dirac bundle $\pi^*(S_N\grtensor S_V)\to V$, i.\,e.\ the canonical extension of $D_M|_U$ to all of $V$. (We do not denote it by $D_V$ here, because we have already used the latter notation differently.)
We may therefore assume that $M=V$ and $S_M=\pi^*(S_N\grtensor S_V)$ everywhere. Then $i$ is an isometry onto its image and $i_*$ as defined in \eqref{eq_inclusion_N_to_M} becomes the canonical isomorphism.
Moreover $\tau=[\beta]$, simplifying the calculations.

The composition product $\llbracket D_M\midd N;\Cl_{0,r}\rrbracket\circ \tau$ is represented by the composition of the comultiplication $\Delta\colon\cS\to\cS\grtensor\cS$ with
\begin{align*}
\cS\grtensor\cS&\xrightarrow{\id\grtensor \beta}\cS\grtensor \cAz(N\Subset V;\Cl_{0,r})\to\fA(\Roe(N\subset V;S_M\grtensor\Cl_{0,r}))
\\\varphi\grtensor\vartheta&\mapsto\left[t\mapsto (\varphi(t^{-1}D_M)\grtensor\id_{\pi^*S'_V})\circ(\id_{S_M}\grtensor\vartheta(C))\right]\,.
\end{align*}
Note that the operator appearing in this formula operates on $L^2(S_M\grtensor\pi^*S'_V)$, which sits inside $L^2(S_M)\grtensor\Cl_{0,r}\grtensor\grelltwo$ and hence we indeed end up in the correct Roe algebra. Because of the continuity in \Cref{lem:parallelorthogonal}\ref{enum:parallelorthogonalseparatecontinuity} and the second part of \Cref{lem:parallelorthogonal}\ref{enum:parallelorthogonaluvformula}, the same $\EE$-theory class is then also represented by the composition of $(\Delta\grtensor\id_\cS)\circ\Delta=(\id_\cS\grtensor\Delta)\circ\Delta$ with 
\begin{align*}
\cS\grtensor\cS\grtensor\cS&\to\fA^2(\Roe(N\subset V;S_M\grtensor\Cl_{0,r}))
\\\varphi\grtensor\psi\grtensor \vartheta&\mapsto\left[s\mapsto \left[t\mapsto(\varphi(s^{-1}D_\|)\psi(t^{-1}D_\perp)\grtensor\id_{\pi^*S'_V})\circ(\id_{S_M}\grtensor\vartheta(C))\right]\right]\,.
\end{align*}

Now we note that applying the calculations in \cite[Sections 1.12 \& 1.13]{guentnerhigson} fiberwise, the composition of $\Delta$ with the asymptotic morphism
\begin{align*}
\psi\grtensor \vartheta&\mapsto \left[t\mapsto(\psi(t^{-1}D_\perp)\grtensor\id_{\pi^*S'_V})\circ(\id_{S_M}\grtensor\vartheta(C))\right]
\end{align*}
is 1-homotopic to the $*$-homomorphism
\(\psi\mapsto \psi(0)(\id_{\pi^*S_N}\grtensor P\grtensor Q)\),
where $P$ denotes the fiberwise orthogonal projection onto the subspace of $L^2(V_x)$ spanned by the function $\zeta_1\colon w\mapsto \exp(-\frac12\|w\|^2)$ and $Q\in\Gamma(\Lin_{\Cl_{r,r}}(S_V\grtensor S'_V))$ denotes the fiberwise orthogonal projection onto a section of the form $\zeta_2\cdot\Cl_{r,r}$, where $\zeta_2$ is a parallel section of the evenly graded part of $S_V\grtensor S'_V$.
Indeed, the main difference is that Higson and Guentner have defined the Clifford operator $C$ and the Dirac operator $D$ via anticommuting representations of $\Cl_{0,r}$ and $\Cl_r$ on $\Cl_{0,r}$ seen as a graded Hilbert space, whereas we are working with $\Cl_{r,r}$-linear Clifford and Dirac operators defined by the canonical anticommuting actions on a principle $\Cl_{r,r}$-bundle. Therefore, we just have to transfer their results by exploiting the equivalence between the categories of graded Hilbert spaces and graded right $\Cl_{r,r}$-modules.

Since $\zeta_1$ is clearly horizontally parallel and $\zeta_2$ is parallel, \Cref{lem:parallelorthogonal}\ref{enum:parallelorthogonalhorizontalpart} implies that $D_\|$ acts on the image of $\id_{S_N}\grtensor P\grtensor Q$ as $D_N\grtensor\id$. We now collect all the parts to see that the asymptotic morphism representing the composition product $\llbracket D_M\midd N;\Cl_{0,r}\rrbracket\circ \tau$ is $1$-homotopic to the tensor product of the asymptotic morphism
$\varphi\mapsto [t\mapsto \varphi(t^{-1}D_N)]$ and a projection onto a canonical trivial irreducible $\Cl_{r,r}$-submodule.
Therefore, $\ind(D_N)$ is being mapped to $\llbracket D_M\midd N;\Cl_{0,r}\rrbracket\circ \tau$ under the canonical map
\[\K(\Roe(N;A))\xrightarrow{i_*}\K(\Roe(N\subset M; A))\cong\K(\Roe(N\subset V; A\grtensor\Cl_{r,r}))\,.\qedhere\]
\end{proof}

\subsection{Relaxing the assumptions}

We want to relax the strong \cref{assump_thom_one,assump_thom_two,assump_thom_three} in \cref{thm:Thomclassrelation}. As it turns out, manifolds of bounded geometry with uniformly embedded submanifolds are a good starting point.

\begin{defn}[{cf.\ \cite[Chapter 2]{eldering_normallyhyperbolic}}]\label{defn:boundedgeometry}
Let $M$ be a Riemannian manifold and $N\subset M$ a submanifold.
\begin{enumerate}
\item\label{it:boundedgeometrymanifold} We say that $M$ has \emph{bounded geometry} on an open subset $U$ if the injectivity radius is uniformly positive on $U$ and the Riemannian curvature tensor as well as all its covariant derivatives are uniformly bounded on $U$.

\item \label{it:uniformlyembedded} The submanifold $N\subset M$ is \emph{uniformly embedded} if there exists an $\delta>0$ such that for all $x\in M$ the intersection $N\cap B_\delta(x)$ is represented in normal coordinates on $B_\delta(x)\subset M$ by the graph of a smooth function $h_x\colon T_xM\to V_x$, the family of functions $\{h_x\}_{x\in M}$ are equi-uniformly bounded and the latter also holds for the families of each of their derivatives.

\item\label{it:boundedgeometrybundle} A $B$-bundle over $M$ equipped with a metric connection is said to have \emph{bounded geometry} if its curvature tensor as well as all its covariant derivatives are uniformly bounded. 
\qedhere
\end{enumerate}
\end{defn}

The notion of bundles of bounded geometry was actually defined slightly differently in \cite[Definition 2.14]{eldering_normallyhyperbolic} and \cite[Appendix 1.1]{shubin}, namely via the existence of certain preferred trivializations. For bundles over manifolds of bounded geometry, their definition agrees with our \ref{it:boundedgeometrybundle} above, as one can readily verify by adapting the proof of the equivalence between b) and b') of \cite[Definition 1.1 in Appendix 1.1]{shubin} from tangent bundles to more general bundles.

\begin{lem}\label{lem:boundedgeometryproperties}
Let $M$ be a complete Riemannian manifold and $N\subset M$ a complete submanifold with normal bundle $V$. We assume that $N$ is uniformly embedded into $M$ and $M$ has bounded geometry in some $R$-neighborhood of $N$. Then:
\begin{enumerate}
\item The submanifold $N$ with the restricted Riemannnian metric has bounded geometry. 
\item\label{item_implies_tubed} There exists $0<R_1<R$ such that the exponential map of $M$ restricts to a diffeomorphism $\Xi$ from the $R_1$-disc bundle in $V$ onto a tubular neighborhood of $N$ in $M$ such that $\Xi$ and $\Xi^{-1}$ as well as each of their derivatives are uniformly bounded.
\item\label{it:normalbundleboundedgeometry} The normal bundle $V$ has bounded geometry.
\end{enumerate}
\end{lem}

\begin{proof}
The first two points are just \cite[Lemma 2.27, Theorem 2.31]{eldering_normallyhyperbolic} where we note that we do not need bounded geometry on all of $M$, but having it on an $R$-neighborhood of $N$ is clearly sufficient.
To prove the third point, calculate the curvature of $V$ and its derivatives in terms of the derivatives of $\Xi$ and apply \ref{item_implies_tubed}.
\end{proof}

\begin{cor}\label{cor:boundedgeometrytoassumptionone}
Given $M,N,V,R,R_1,\Xi$ as in the lemma and $0<R_2<R_1$, we can equip $M$ with a new complete Riemannian metric which agrees with the old one on $N$ and outside of the $R_1$-neighborhood of $N$ and such that $\Xi$ is an isometry from the $R_2$-disk bundle in $V$ onto the $R_2$-neighborhood of $N$ in $M$ equipped with the new metric. 
In other words: The manifold $M$ equipped with the new Riemannian metric satisfies \Cref{assump_thom_one} for the $R_2$-disk bundle.

Moreover, the old and new Riemannian metric are bi-Lipschitz equivalent
and hence have the same Roe algebras and $\cA$-function algebras.
\end{cor}

\begin{proof}
The diffeomorphism $\Xi$ transfers the Riemannian metric on the $R_1$-disk bundle in $V$ to a Riemannian metric on the $R_1$-neighborhood of $N$ in $M$. Now use a smooth cut-off function $M\to[0,1]$ which only depends on the distance from $N$, is one on the $R_2$-neighborhood of $N$ and vanishes outside of the $R_1$-neighborhood to interpolate between this metric and the old metric on $M$. The metric we obtain has the desired properties.
\end{proof}

While we have formulated the last definition and statements for arbitrary manifolds, let us now return to the set-up presented at the very beginning of this \Cref{sec:Thomclassrelation} and not repeat it as a prerequisit in every lemma.

\begin{lem}\label{lem:assumptiononetoassumptiontwo}
If \Cref{assump_thom_one} is satisfied, i.\,e.\ the tubular neighborhood $U$ is isometric to the $R_2$-disc bundle in $V$ for some $R_2>0$, and $S_M|_N$ with its restricted action of $\Cl(TM)|_N\cong\Cl(TN)\grtensor\Cl(V)$ is isomorphic to $S_N\grtensor S_V$, then the connection on $S_M$ can be modified over $U$ in such a way that the Dirac bundle $S_M$ with this new connection is isomorphic to the Dirac bundle $\pi^*(S_N\grtensor S_V)$ on a slightly smaller disc bundle $U'$ of radius $0<R_3<R_2$, i.\,e. \Cref{assump_thom_two} is satisfied on $U'\subset U$.
\end{lem}

\begin{proof}
We extend the given bundle isomorphism $S_N\grtensor S_V\to S_M|_N$ by parallel transport along the radial lines in $V$ to a bundle isomorphism $F\colon \pi^*(S_N\grtensor S_V)|_U\to S_M|_U$. 
Let $X\colon [0,1]\to \Cl(TU)$ and $\xi\colon [0,1]\to \pi^*(S_N\grtensor S_V)$ be parallel sections along such a radial line $\gamma\colon [0,1]\to U$ starting in $N$.
Then the section $F\circ \xi$ of $S_M|_U$ along $\gamma$ is also parallel by choice of $F$ and the compatibility of Clifford action with connection on a Dirac bundle imply that $X\cdot \xi$ and  $X\cdot (F\circ \xi)$ are parallel as well.
By assumption, $F$ preserves the Clifford action over $N$, so the initial values $F(X(0)\cdot \xi(0))=X(0)\cdot F(\xi(0))$ agree and we conlcude $F(X\cdot\xi)=X\cdot (F\circ \xi)$.
We have thus shown that the whole bundle isomorphism preserves the Clifford actions.

Now we can use $F$ to transfer the connection of $\pi^*(S_N\grtensor S_V)$ and obtain a connection $\tilde \nabla^M$ on $S_M|_U$ which, just like $\nabla^M$, is compatible with the Clifford action. Choose a smooth cutoff function $\theta\colon M\to [0,1]$ which is one on $U'$ and vanishes outside of $U$. Then $\theta\cdot\tilde\nabla^M+(1-\theta)\cdot\nabla^M$ defines a new connection on $S_M$ which clearly is compatible with the Clifford action, too. 
With this connection, the claim holds.
\end{proof}

The connection on $V$ together with the choice of the connection on $P_{\unitary_1}(V)$ also gives rise to a $\Cl_{0,r}$-linear connection on $S'_V$ which is compatible with the action of $\Cl(-V)$, the latter equipped with the induced connection.

\begin{lem}\label{lem_thom_N_bounded_geom}
If the submanifold $N$ has bounded geometry and $S'_V$ is a bundle of bounded geometry, then there exists an isometric inclusion as a direct summand of bounded variation of a trivial bundle
\[G\colon S'_V\hookrightarrow (\Cl_{0,r}\grtensor\C^k)\times N\subset (\Cl_{0,r}\grtensor\grelltwo)\times N\]
such that all derivatives of $G$ as well as its dual projection $(\Cl_{0,r}\grtensor\C^k)\times N\to S'_V$ with respect to synchronous framings are bounded.
As a consequence, the image of $\beta$ under the inclusion into $\Cb(V;\Cl_{0,r}\grtensor\C^{k\times k})$ is contained in $\cAz(N\Subset V;\Cl_{0,r})$ and hence \Cref{assump_thom_three} is satisfied.
Moreover, the inclusion and therefore also $\beta\colon\cS\to \cAz(N\Subset V;\Cl_{0,r})$ is unique up to homotopy. \end{lem}

If $M$ has bounded geometry in an $R$-neighborhood of $N$, then the normal bundle $V$ also has bounded geometry by \Cref{lem:boundedgeometryproperties}\ref{it:normalbundleboundedgeometry}, but this does not automatically imply that $S'_V$ has bounded geometry. Indeed, $S'_V$ will have bounded geometry if and only if the line bundle associated to $P_{\unitary_1}(V)$ has bounded geometry. This is of course satisfied if the \spinC{} structure comes from a spin structure, i.\,e.\ if the normal bundle is $\KO$-oriented.

\begin{proof}
The embedding as direct summand of bounded variation is done for this bundle of $\Cl_{0,r}$-modules exactly as in \cite[Prop. 3.17]{engel_uniform_Kth} for vector bundles.

Since the connection on $S'_V$ is compatible with the action of $\Cl(-V)$, for each smooth $\varphi\in\cS_R$ supported in $[-R,R]$ the section $\beta(\varphi)$ is a smooth section of $\pi^*\Lin_{\Cl_{0,r}}(S'_V)$ which is supported in the $R$-neighborhood of $N$, horizontally parallel and has uniformly bounded derivatives in vertical direction.
Therefore the section $G \beta(\varphi) G^*$, which is the image of $\beta(\varphi)$ under the inclusion into $\Cb(V;\Cl_{0,r}\grtensor\C^{k\times k})$, also has bounded derivatives and is supported in the $R$-neighborhood of $N$. Such functions $\varphi$ are dense in $\cS$ and we conclude that the image of $\beta$ will be contained in $\cAz(N\Subset V;\Cl_{0,r})$ under the inclusion.

Finally, if $G,G'\colon S'_V\hookrightarrow (\Cl_{0,r}\grtensor\C^k)\times N$ are two such inclusions, then $G_s\coloneqq \cos(s)G\oplus\sin(s)G'$ defines a homotopy of inclusions of the same type between $G\oplus 0$ and $0\oplus G'$. The resulting maps $G_s\colon \Cb(N;S'_V)\to\Cb(N(\Cl_{0,r}\grtensor\grelltwo))$ clearly depend continuously in operator norm on $s$
and thus the induced inclusions $\Cb(V;\pi^*\Lin_{\Cl_{0,r}}(S'_V))\subset \Cb(V;\Cl_{0,r}\grtensor\grKom)$ are homotopic as $*$-homomorphisms.
\end{proof}

\begin{defn}\label{def:Thomclassboundedgeometry}
Assume that $M$ has bounded geometry in an $R$-neighborhood of $N$, $N\subset M$ is uniformly embedded and the bundle $S'_V$ has bounded geometry. Then \Cref{lem:boundedgeometryproperties} implies that $N$ has bounded geometry and yields the diffeomorphism $\Xi$ from the $R_1$-disk bundle $U$ in $V$ to the $R_1$-neighborhood $W\subset M$ of $N$. \Cref{lem_thom_N_bounded_geom} applies and we obtain a $*$-homomorphism
\[\cS\xrightarrow{\psi_{R_1}}\cS_{R_1}\xrightarrow{\beta|_{\cS_{R_1}}}\cA(\overline{U},\partial U;\Cl_{0,r})\xrightarrow{(\Xi^{-1})^*}\cA(\overline{W},\partial W;\Cl_{0,r})\subset \cA(N\Subset M;\Cl_{0,r})\]
which is unique up to homotopy. We define the \emph{Thom class of $N$ in $M$} as the uniquely determined element in $\K(\cA(N\Subset M;\Cl_{0,r}))$ represented by this $*$-homomorphism.
\end{defn}

It is obvious that if $M$ also fullfills \Cref{assump_thom_one}, then $\Xi$ will be the identity and therefore the Thom class defined here agrees with the one defined in \Cref{def:Thomclassassumptions}.

\begin{thm}\label{thm_Thomclassrelation_relaxed}
Let $M$ be a complete Riemannian manifold and $N\subset M$ a complete Riemannian submanifold with $\K$-oriented normal bundle $V$. We assume that $N$ is uniformly embedded into $M$, $M$ has bounded geometry in some $R$-neighborhood of $N$ and that the spinor $\Cl_{0,r}$-bundle $S'_V$ associated to $V$ has bounded geometry.
Furthermore we assume that $S_M|_N$ with its restricted action of $\Cl(TM)|_N\cong\Cl(TN)\grtensor\Cl(V)$ is isomorphic to $S_N\grtensor S_V$. Then:
\begin{equation}
\label{eq_Thom_final}
\llbracket D_M\midd N;\Cl_{0,r}\rrbracket\circ \tau=i_*\ind(D_N)
\end{equation}
\end{thm}

\begin{proof}
First we modify the Riemannian metric on $M$ according to \Cref{cor:boundedgeometrytoassumptionone}, such that \Cref{assump_thom_one} is satisfied, and modify the Clifford action and connection on $S_M$ accordingly. Afterwards we modify the connection on $S_M$ as in \Cref{lem:assumptiononetoassumptiontwo} such that \Cref{assump_thom_two} is also satisfied.
These modifications do not change the Roe algebras and $\cA$-function algebras and the new Dirac operator of the modified bundle $S_M$ is clearly bordant to the old Dirac operator.
Therefore, an analogue of \Cref{thm:BordismInvariance} for the localized $\EE$-theory classes, which can be shown by adapting the proof in \cite[Theorem 4.12]{WulffTwisted}, implies that the class
\[\llbracket D_M\midd N;\Cl_{0,r}\rrbracket\in\EE(\cA(N\Subset M;\Cl_{0,r}),\Roe(N\subset M;A\grtensor\Cl_{r,r}))\]
remains the same.

While applying \Cref{lem:assumptiononetoassumptiontwo} in order to satisfy \Cref{assump_thom_two}, we might have changed the connection on $S_N$. However, the Dirac operator $D_N$ associated to the new connection is also bordant to the one with respect to the old connection, so $\ind(D_N)$ also stays unchanged.

Finally, the change of Riemannian metric does also not change the Thom class, because it just turns the representative from \Cref{def:Thomclassboundedgeometry} to the one from \Cref{def:Thomclassassumptions}.
Therfore, after these modifications we can simply apply \Cref{thm:Thomclassrelation} and the claim follows.
\end{proof}

\section{The index away from \texorpdfstring{$L$}{L}}
\label{sec:indexawayfromL}

In this section we present a direct generalization of the theory presented in \cite[Sections 7 \& 8]{WulffTwisted} about coarse indices of operators which are defined over complements of compact subsets.

\subsection{Definition and calculations using \texorpdfstring{$\EE$}{E}-theory}
We consider an $A$-linear Dirac operator $D$ over the complement of a controlled neighborhood of the closed subset $L$ in $M$ and want to define its index in the $\K$-theory of the relative Roe algebra, whose definition we recall right away.

\begin{defn}
The \emph{relative Roe algebra} of a closed subset $L$ in a proper metric space $X$ is the quotient \textCstar-algebra 
\[\Roe(X,L;A)\coloneqq \Roe(X;A)/\Roe(L\subset X;A)\]
(or $\Roe(X,L;\Hilbert_X)\coloneqq \Roe(X;\Hilbert_X)/\Roe(L\subset X;\Hilbert_X)$, if the choice of Hilbert module is relevant). 
\end{defn}

Now, by choosing a uniformly close smooth approximation of the distance function $\dist(\blank,L)\colon M\to[0,\infty)$ and picking a sufficiently large regular value, we find a closed submanifold with boundary of codimension zero $W$ containing the set on which $D$ is not defined and such that the inclusion $L\hookrightarrow W$ is a coarse equivalence.
Then the Dirac operator $D$ over $V\coloneqq M\setminus \mathring W$ extends to an $A$-linear Dirac operator $\hat D$ over a complete Riemannian manifold $\hat M$ which contains an Riemann-isometrically embedded copy of $V$ and such that the identity map on $V$ extends to a coarse map $\hat M\to M$.

Indeed, the manifold $\hat M$ can for example be chosen as the double of $V$ and the map $\hat M\to M$ collapses the two halves. We have adressed how to construct the operator $\hat D$ in a similar context somewhat more detailed in \Cref{sec:DefRelInd}. It involves changing the Riemannian metric in a 1-neighborhood of $\partial W$ and the Dirac bundle accordingly, which does not change the coarse indices under consideration.
If $D$ is defined on all of $M$ from the beginning, then we can also simply choose $\hat M=M$.

Since $M$ is a complete Riemannian manifold, it is straightforward to see that the decomposition $M=V\cup W$ into the two submanifolds $V,W$ with common boundary $\partial V=\partial W$ is coarsely excisive. Hence we have
\begin{align*}
\Roe(M,L;A)&=\frac{\Roe(M;A)}{\Roe(L\subset M;A)}=\frac{\Roe(V\subset M;A)+\Roe(W\subset M;A)}{\Roe(W\subset M;A)}
\\&\cong\frac{\Roe(V\subset M;A)}{\Roe(V\subset M;A)\cap\Roe(W\subset M;A)}=\frac{\Roe(V\subset M;A)}{\Roe(\partial V\subset M;A)}
\end{align*}
and similarly we obtain an isomorphism 
\[\Roe(\hat M,\hat M\setminus\mathring V;A)\cong \Roe(V\subset \hat M;A)/\Roe(\partial V\subset \hat M;A)\,.\]
The coarse map $\hat M\to M$ therefore gives rise to a $*$-homomorphism
\[\Roe(\hat M;A)\to \Roe(\hat M,\hat M\setminus\mathring V;A)\to\Roe(M,L;A)\,.\]

\begin{defn}
The \emph{coarse index away from $L$} of the operator $D$ is the image 
$\ind(D,  L)\in \K(\Roe(M,L;A))$
of the coarse index $\ind(\hat D)$ under the canonical map induced by the $*$-homomorphism described above.
\end{defn}

We will see in a moment that the coarse index away from $L$ is independent of the choice of extension $\hat M,\hat D$.
In the case of compact $L$ it is exactly the index $\ind_{/\Kom}(D)$ defined in \cite[Section 7]{WulffTwisted}.

If $D$ is defined everywhere on $M$, then $\ind(D, L)$ is the image of $\ind(D)$ under the canonical map $\K(\Roe(M;A))\to \K(\Roe(M,L;A))$ and is therefore an obstruction to invertibility of $D$ away from $L$. Applied to the Dirac operator of a spin structure on $M$ we get:
\begin{cor}\label{cor_upsc_outside_L}
Let $M$ be an $m$-dimensional complete Riemannian spin manifold without boundary and denote by $\slashed{D}$ its spin Dirac operator.

If $M$ has uniformly positive scalar curvature outside of a controlled neighbourhood of $L$, then $\ind(\slashed{D},L) = 0$.
\end{cor}
\begin{proof}
In \cref{ex_invertible_upsc_outside_L} we have discussed that in this situation $\slashed{D}$ will be invertible away from $L$. From the discussion surrounding the Diagram~\eqref{eq_diag_Lloc_classes} we conclude that $\ind(\slashed{D}) \in \K(\Roe(M;\Cl_m))$ has a lift to a class $\ind(\slashed{D}\mid L) \in \K(\Roe(L \subset M;\Cl_m))$ and therefore $\ind(\slashed{D},L) = 0$.
\end{proof}

The coarse index away from $L$ can also be computed using an $\EE$-theory class.
Formula \eqref{eq:DiracEclassAsympMorphSubset} for $\hat D,\hat M, V,\partial V$ instead of $D,M,\overline{U},\partial U$ provides us with an asymptotic morphism
\begin{align*}
\cS\grtensor \cA(V,\partial V;B)&\to\fA( \Roe(V^{\subseteq \hat M};A\grtensor B))\end{align*}
which represents the $\EE$-theory class $\llbracket \hat D|_{\mathring V};B\rrbracket$. After composing it with the canonical inclusion $\Roe(V^{\subseteq \hat M};A\grtensor B))\subset \Roe(V^{\subseteq M};A\grtensor B)$, the class becomes independent of the chosen extension by \Cref{lem:independence}.
Even more, the proof of the Lemma given in  \cite[Theorem 4.9]{WulffTwisted} shows that this independence already holds at the level of asymptotic morphisms.
If we restrict the domain to the ideal 
\begin{align*}
\cAz(\partial V\Subset V,\partial V\mid L;B)&\coloneqq \cA(V,\partial V;B)\cap\cAz(\partial V\Subset V;B)
\\&\subset \cA(V,\partial V;B)\subset \cA(M;B)\,,
\end{align*}
see \Cref{defn_A0_subsubset} for notation,
then we end up in the ideal $\Roe(\partial V\subset V^{\subseteq M};A\grtensor B)\subset \Roe(V^{\subseteq M};A\grtensor B)$. It is straightforward to see the canonical isomorphisms
\begin{align*}
\Roe(M,L;A\grtensor B)&\cong\frac{\Roe(V^{\subseteq M};A\grtensor B)}{\Roe(\partial V\subset V^{\subseteq M};A\grtensor B)}
\\\cA(M\div L;B)&\coloneqq\frac{\cA(M;B)}{\cAz(L\Subset M;B)}\cong \frac{\cA(V,\partial V;B)}{\cAz(\partial V\Subset V,\partial V;B)}\,.
\end{align*}
The function algebra $\cA(M\div L;B)$ should be seen as relative to $L$ in a coarse geometric sense, but we do not denote it by $\cA(M, L;B)$, because the latter could also mean the subalgebra of $\cA(M;B)$ of functions vanishing on $L$.
We end up with the following commutative diagram with exact columns:
\begin{equation}\label{eq:SESAsympMorph}
\xymatrix{
0\ar[d]&0\ar[d]
\\\cS\grtensor \cAz(\partial V\Subset V,\partial V;B)\ar[r]\ar[d]&\fA(\Roe(\partial V\subset V^{\subseteq M};A\grtensor B))\ar[d]
\\\cS\grtensor \cA(V,\partial V;B)\ar[r]\ar[d]&\fA(\Roe(V^{\subseteq M};A\grtensor B))\ar[d]
\\\cS\grtensor \cA(M\div L;B)\ar[r]\ar[d]&\fA(\Roe(M,L;A\grtensor B))\ar[d]
\\0&0
}\end{equation}
The first two horizontal arrows represent the $\EE$-theory classes
$\llbracket D|_{\mathring V}\midd \partial V;B\rrbracket$ and $\llbracket D|_{\mathring V};B\rrbracket$, respectively.
\begin{defn}
The $\EE$-theory class away from $L$ of $D$ is the class
\[\llbracket D, L;B\rrbracket\in \EE( \cA(M\div L;B),\Roe(M,L;A\grtensor B)) \]
represented by the bottom horizontal arrow in Diagram \eqref{eq:SESAsympMorph}.
\end{defn}
All three of these $\EE$-theory classes are independent of the chosen extension $\hat D,\hat M$.

\begin{rem}
There is undoubtedly a linguistic inconvenience in naming the various indices, \textCstar-algebras and $\EE$-theory classes. All of them are justifiable relative in one way or another, but with respect to which subsets?
For example, the relative Roe algebra $\Roe(M,L;A)$, the index $\ind(D,L)$ away from $L$ and the corresponding $\EE$-theory class $\llbracket D,L;B\rrbracket$ can be seen as relative to $L$, because the information above these subspaces is neglected or cancelled out. In contrast, the generalizations of the relative coarse indices $\ind(D\mid L)$ and $\ind(D\midd E,F)$ as well as their corresponding $\EE$-theory classes should rather be interpreted as ``relative to a coarse geometric complement of L'' or, in other words, ``local/localized over $L$''.

One should thus keep close attention to the different notations and namings and accept that they partially have historical roots and might not be appropriate depending on the perception in one's mind.
\end{rem}

It is straightforward to directly generalize \cite[Theorem 7.10]{WulffTwisted} and its proof to obtain the following theorem.
\begin{thm}\label{thm:indexawayfromLEtheoryproduct}
Let $E$ be a $B$-bundle of bounded variation defined over the complement of a controlled neighborhood of $L$ in $M$, which clearly has a $\K$-theory class
\[
\llbracket E\div L\rrbracket\in\K(\cA(M\div L;B))\,.
\]
Then the index away from $L$ of the twisted operator $D_E$, which is also defined on the complement of a controlled neighborhood of $L$, is given by the formula
\[
\pushQED{\qed} \ind(D_E,L)=\llbracket D,L;B\rrbracket\circ\llbracket E\div L\rrbracket\,.\qedhere
\popQED
\]
\end{thm}
In particular, for the the constant rank one vector bundle $E$ it calculates $\ind (D,L)$, showing that the latter is also independent of the choice of extension.

We conclude this subsection by discussing the compatibility of the $\EE$-theory classes of $D$ with boundary morphisms in $\EE$-theory.
The short exact sequences of Diagram \eqref{eq:SESAsympMorph} without the ``decorations'' $\cS\grtensor\blank$ and $\fA(\blank)$ give rise to boundary $\EE$-theory classes
\begin{align*}
\llbracket\sigma_0\rrbracket &\in \EE(\Cz(0,1)\grtensor \cA(M\div L;B), \cAz(\partial V\Subset V,\partial V;B))
\\\llbracket\sigma_1\rrbracket &\in \EE(\Cz(0,1)\grtensor \Roe(M,L;A\grtensor B),\Roe(\partial V\subset V^{\subseteq M};A\grtensor B))
\end{align*}
For separable \textCstar-algebras this would be \cite[Section 5]{guentnerhigsontrout}, but since we are dealing with non-separable \textCstar-algebras, we have to use the $\EE$-theory model presented in \cite{Wulff_NonseperableEtheory} which fixes the problems with exactness, see \cite[Section 7]{Wulff_NonseperableEtheory}.
Moreover, the tensor factor $\Cz(0,1)$ can be replaced by $\Cl_1$, because these graded \textCstar-algebras are canonically isomorphic in the $\EE$-theory category, and since they are finite dimensional it is obvious that they can be incorporated into the coefficient \textCstar-algebra $B$ as in the following proposition.

\begin{prop}\label{prop:Eclassesboundarymorphisms}
The following diagram commutes in the $\EE$-theory category:
\[\xymatrix{
\cA(M\div L;B\grtensor \Cl_1)\ar[r]^{\llbracket\sigma_0\rrbracket}\ar[d]^{\llbracket D,L; B\grtensor\Cl_1\rrbracket} & \cAz(\partial V\Subset V,\partial V;B))\ar[d]^{\llbracket D|_{\mathring V}\midd\partial V;B\rrbracket}
\\\Roe(M,L;A\grtensor B\grtensor\Cl_1)\ar[r]^{\llbracket\sigma_1\rrbracket}&\Roe(\partial V\subset V^{\subseteq M};A\grtensor B)
}\]
\end{prop}

\begin{proof}
This follows from \cite[Proposition 7.6]{Wulff_NonseperableEtheory} (cf.\ \cite[Proposition 5.8]{guentnerhigsontrout}). The additional ``decoration'' in \eqref{eq:SESAsympMorph} is readily seen to disappear once one passes from the simpletotic category to the $\EE$-theory category.
\end{proof}

\subsection{Bundles of vanishing variation}

In the previous sections we have seen that indices of twisted operators $D_E$ can be calculated from $\EE$-theory classes of $D$ by composition products.
It is in general not possible to calculate the index of $D_E$ from the index of $D$ alone. Indeed, there are examples even in classical Fredholm index theory on compact manifolds that the index of $D$ vanishes while the index of $D_E$ does not.

The big insight of \cite[Section 8]{WulffTwisted} is that for special types of bundles $E$, the bundles of vanishing variation, it is possible to calculate the index $\ind(D_E,L)$ from $\ind(D,L)$ via a module multiplication (which should be seen as a cap product) if $L$ is compact. We are now going to generalize this theory to arbitary closed subsets $L\subset M$.

We first need to introduce the generalization of the stable Higson compactification and corona.
There are two ways of introducing it, a topological one in the spirit of \cite[Definition 8.1]{WulffTwisted} and an analytic one. 
They agree for complete Riemannian manifolds of bounded geometry by a direct generalization of \cite[Lemma 8.2]{WulffTwisted}.
Here we opt for the shorter analytic definition out of convenience, because it is the one which is relevant to our index theoretic considerations.
We leave it to the reader to elaborate the topological approach which works for all proper metric spaces, noting that it might become relevant in future papers, for example to analyze properties of the generalized co-assembly map which we will introduce below.

\begin{defn}\label{defn_stable_Higson_away}
Let $M$ be a complete Riemannian manifold, $L\subset M$ a closed subspace and $B$ a graded \textCstar-algebra.
\begin{itemize}
\item We say that a smooth function $f\colon M\to B\grtensor\grKom$ has \emph{vanishing variation away from $L$} if for all $\varepsilon>0$ there is a controlled neighborhood $U$ of $L$ such that $\|\grad f(x)\|<\varepsilon$ for all points $x$ outside of $U$.

\item The \emph{stable Higson compactification of $M$ away from $L$ with coefficients in $B$} is the norm closure $\sHigCom_L(M;B)$ of
\[\sHigCom^\infty_L(M;B)\coloneqq\{f\in\cA^\infty(M;B)\mid f \text{ has vanishing variation away from } L\}\,.\]
in $\cA(M;B)$.
It is a \textCstar-algebra containing $\cAz(L\Subset M;B)$ as an ideal and we define the \emph{stable Higson corona of $M$ away from $L$} as the quotient \textCstar-algebra
\[\sHigCor_L(M;B)\coloneqq \frac{\sHigCom_L(M;B)}{\cAz(L\Subset M;B)}\subset\cA(M\div L;B)\,.\qedhere\]
\end{itemize}
\end{defn}

\begin{rem}\label{rem_bounded_variation_need}
It is debatable if one should really include the condition of bounded variation $f\in\cA^\infty(M;B)$ into the definition as we have done, or if $f\in\Cb^\infty(M;B\grtensor\grKom)$ plus vanishing variation away from $L$ is more appropriate, but since we only work with $\cA$-functions in this article, it seemed adequate to add it.
\end{rem}

\begin{defn}[{cf.\ \cite[Definition 8.4]{WulffTwisted}}]
A \emph{$B$-bundle of vanishing variation away from $L$} is a $B$-bundle defined on the complement of a controlled neighborhood of $L$ in $M$ whose fiber at a point $x$ is the image of a function $P\in\sHigCom^\infty_L(M;B)$ which is projection valued outside of the controlled neighborhood.
The defining function $P$ determines a class $\llbracket E,\sHigCor_L\rrbracket\in \K(\sHigCor_L(M;B))$ which is a preimage of $\llbracket E\div L\rrbracket$ under the canonical map.
\end{defn}

The following main results of this section essentially work exactly as the corresponding statements in \cite[Section 8]{WulffTwisted} and hence we only have to give little details of the proofs.

\begin{lem}[{cf.\ \cite[(8.1), Lemma 8.5, Corollary 8.6]{WulffTwisted}}]
Let $M$ be a complete Riemannian manifold of bounded geometry, $L\subset M$ a closed subset and $A,B,C$ graded \textCstar-algebras.
\begin{enumerate}
\item There is a $*$-homomorphism
\[\nabla_{B,C}\colon\sHigCor_L(M;B)\grtensor\sHigCor_L(M;C)\to\sHigCor_L(M;B\grtensor C)\]
which multiplies functions pointwise via a $*$-isomorphism $A\grtensor \grKom\grtensor B\grtensor\grKom\cong A\grtensor B\grtensor\grKom$, and a $*$-homomorphism
\[\fm_{A,B}\colon\Roe(M,L;A)\grtensor\sHigCor_L(M;B)\to\Roe(M,L;A\grtensor B)\,,\quad [T]\grtensor[f]\to[Tf]\]
where the operator $T\in\Roe(M;A)\subset\Lin(L^2(M)\grtensor A\grtensor\grelltwo)$ and the function $f\in\sHigCom_L(M;B)\subset\Cb(M;B\grtensor\grelltwo)$ act on the Hilbert module $L^2(M)\grtensor A\grtensor \grelltwo\grtensor B\grtensor \grelltwo\cong L^2(M)\grtensor A\grtensor B\grtensor \grelltwo$.

They are canonical up to homotopy, because the isomorphisms $\grelltwo\grtensor\grelltwo\cong\grelltwo$ and $\grKom\grtensor\grKom\cong\grKom$ are.

\item These $*$-homomorphisms induce a cup product 
\[\cup\colon\K(\sHigCor_L(M;B))\otimes\K(\sHigCor_L(M;C))\to\K(\sHigCor_L(M;B\grtensor C))\]
and a cap product
\[\cap\colon\K(\Roe(M,L;A))\otimes\K(\sHigCor_L(M;B))\to\K(\Roe(M,L;A\grtensor B))\]
which are associate with respect to each other, i.\,e.\ for all $x,y,z$ we have
\[
x\cap(y\cup z)=(x\cap y)\cap z\,.
\]
\end{enumerate}
\end{lem}

\begin{proof}
That $\nabla_{B,C}$ is a $*$-homomorphism is clear. 

For $\fm_{A,B}$ it is straightforward to see that $Tf,fT\in\Roe(M;A\grtensor B)$ and we have to prove that the commutators $[T,f]$ are contained in the ideal $\Roe(L\subset M;A\grtensor B)$. 
To this end we may assume that $T$ has finite propagation and use bounded geometry to approximate $f$ by a step function $\check f$ just as in the proof of \cite[Lemma 8.5]{WulffTwisted} and then the calculations therein show that $[T,\check f]$ can be approximated by operators supported in controlled neighborhoods of $L$ and the same is true for $T(f-\check f)$ and $(f-\check f)T$. Hence we obtain a $*$-homomorphism
\[\Roe(M;A)\grtensor\sHigCom_L(M;B)\to\Roe(M,L;A\grtensor B)\]
and it is again straightforward to see that $Tf$ is supported near $L$ if $T$ is or if $f$ vanishes away from $L$. Hence it factors through the claimed quotient algebras.

For the associativity of the products one proves as in \cite[Lemma 8.5]{WulffTwisted} that $\fm_{A,B\grtensor C}\circ(\id\grtensor \nabla_{B,C})$ and $\fm_{A\grtensor B,C}\circ(\fm_{A,B}\grtensor \id)$ are homotopic.
\end{proof}

Finally, the following very interesting index theorem can be proven exactly as \cite[Theorem 8.7]{WulffTwisted}.

\begin{thm}\label{thm:indexcapproduct}
Let $M$ be a complete Riemannian manifold of bounded geometry, $D$ an $A$-linear Dirac operator defined on the complement of some controlled neighborhood of a closed subset $L\subset M$ and $A,B$ graded \textCstar-algebras.
Letting $\iota\colon\sHigCor_L(M;B)\hookrightarrow\cA(M\div L;B)$ denote the inclusion, then
\begin{equation}
\label{eq_cap_with_vanishing_var}
\llbracket D,L;B\rrbracket \circ\iota_*(y)=\ind(D,L)\cap y
\end{equation}
for all $y\in \K(\sHigCor_L(M;B))$.
In particular, for every $B$-bundle $E$ of bounded variation away from $L$ the index of the twisted operator $D_E$ can be calculated by
\[\ind(D_E,L)=\ind(D,L)\cap\llbracket E,\sHigCor_L\rrbracket\,.\qedhere\]
\end{thm}

\section{Application to positive scalar curvature}

In this section we will use our calculations to find obstructions to uniformly positive scalar curvature away from a closed submanifold $N$.
A key ingredient to our main theorem is the following version of the coarse co-assembly map introduced in \cite{EmeMey}. 

\begin{defn}
The \emph{coarse co-assembly map away from $N$} is the boundary map
\[\mu_N\colon \K(\sHigCor_N(M;B\grtensor\Cl_1))\to\K(\cAz(N\Subset M;B)) \]
associated to the short exact sequence 
\[
0\to\cAz(N\Subset M;B)\to\sHigCom_N(M;B)\to\sHigCor_N(M;B)\to 0
\]
from \cref{defn_stable_Higson_away}.
\end{defn}

Here we have implemented the index shift again by tensoring the coefficient \textCstar-algebra by $\Cl_1$ in order to stay in the language of \Cref{prop:Eclassesboundarymorphisms}.

\begin{thm}\label{thm_thom_lifts_application}
Let $M$ be a complete Riemannian $m$-dimensional spin-manifold and $N\subset M$ be a complete $n$-dimensional submanifold with $\K$-oriented normal bundle~$V$. Hence $N$ is at least \spinC\ and we let $\slashed D_M$ and $\slashed D_N$ be the $\Cl_m$- and $\Cl_n$-linear spin (resp. \spinC) Dirac operators of $M$ and $N$, respectively. Let $i\colon N\hookrightarrow M$ denote the inclusion and $r\coloneqq m-n$ the codimension.

Assume that $N$ is uniformly embedded into $M$, that $M$ has bounded geometry in some $R$-neighbourhood of $N$ and that the line bundle associated to $P_{\unitary_1}(V)$ of the \spinC -structure of $N$ has bounded geometry (which is the case if $V$ is $\KO$-oriented).

If the Thom class $\tau$ lies in the image of the coarse co-assembly map away from $N$
\[\mu_N\colon \K(\sHigCor_N(M;\Cl_{1,r}))\to\K(\cAz(N\Subset M;\Cl_{0,r}))\]
with preimage $\tilde\tau$,
then 
the wrong way map 
\begin{align*}
\K(\Roe(M,N;\Cl_m))&\to  \K(\Roe(N\subset M;\Cl_{m,r}))\\
x&\mapsto \partial(x\cap\tilde\tau)
\end{align*}
maps $\ind(\slashed D_M, N)$ to $i_*\ind(\slashed D_N)$.
Thus, $i_*\ind(\slashed D_N)\not=0$ implies that
 there cannot be a metric of uniformly positive scalar curvature away from $N$ in the same quasi-isometry class as the original one.
\end{thm}

\begin{proof}
The coassembly map $\mu_N$ is given by composition with an $\EE$-theory element which we denote by the same letter and then the diagram in \Cref{prop:Eclassesboundarymorphisms} is readily seen to extend to the commutative diagram
\begin{equation}\label{eq_diag_lifts}
\xymatrix{
\sHigCor_N(M;\Cl_{1,r})\ar[d]^{\iota_*}\ar[drr]^{\mu_N}&&
\\
\cA(M\div N;\Cl_{1,r})\ar[r]^-{\llbracket\sigma_0\rrbracket}\ar[d]^{\llbracket \slashed D_M,N; \Cl_{1,r}\rrbracket}
& \cAz(\partial V\Subset V,\partial V;\Cl_{0,r}))\ar[d]^{\llbracket \slashed D_M|_{\mathring V}\midd\partial V;\Cl_{0,r}\rrbracket}\ar@{}[r]|{\subset}
& \cAz(N\Subset M;\Cl_{0,r})\ar[d]^{\llbracket \slashed D_M\midd N;\Cl_{0,r}\rrbracket}
\\
\Roe(M,N;\Cl_{m+1,r})\ar[r]^-{\llbracket\sigma_1\rrbracket} 
& \Roe(\partial V\subset V^{\subseteq M};\Cl_{m,r})\ar@{}[r]|{\subset}
& \Roe(N\subset M;\Cl_{m,r})
}
\end{equation}
where the bottom row composes to yield the boundary $\EE$-theory element $\partial$ associated to the short exact sequence of Roe algebras
\[0\to\Roe(N\subset M;\Cl_{m,r})\to\Roe(M;\Cl_{m,r})\to \Roe(M,N;\Cl_{m,r})\to 0\,.\]

Together with \Cref{thm_Thomclassrelation_relaxed} we can now calculate
\begin{align*}
0& \stackrel{\phantom{\eqref{eq_diag_lifts}}}{\not=} i_*\ind(\slashed D_N) \stackrel{\eqref{eq_Thom_final}}= \llbracket \slashed D_M\midd N;\Cl_{0,r}\rrbracket\circ \tau=\llbracket \slashed D_M\midd N;\Cl_{0,r}\rrbracket\circ \mu_N\circ \tilde\tau
\\
& \stackrel{\eqref{eq_diag_lifts}}= \partial\circ\llbracket \slashed D_M,N;\Cl_{1,r}\rrbracket\circ\iota_*\circ \tilde\tau \stackrel{\eqref{eq_cap_with_vanishing_var}}= \partial(\ind(\slashed D_M,N)\cap\tilde\tau)
\end{align*}
and the second claim follows by \cref{cor_upsc_outside_L}.
\end{proof}

\Cref{thm_thom_lifts_application} has the assumption that the Thom class $\tau$ of the submanifold $N$ in $M$ lifts through the coarse co-assembly map away from $N$
\[\mu_N\colon \K(\sHigCor_N(M;\Cl_{1,r}))\to\K(\cAz(N\Subset M;\Cl_{0,r}))\,.\]
We describe now a situation where we can guarantee this: We treat multi-partitioned manifolds and recover the results from \cite{schick_zadeh} and \cite[Section~4.3]{siegel_analytic_structure_group}.

\begin{defn}
We say that a complete Riemannian manifold $M$ is \emph{multi-partitioned} by a submanifold $N$ if there exists a function $f\colon M \to \R^r$ satisfying:
\begin{enumerate}
\item\label{it:mp:uniformlycontinuous} $f$ is uniformly continuous.
\item\label{it:mp:smoothboundedder} $f$ is smooth on an $R$-neighborhood of $N$ with uniformly bounded derivatives thereon.
\item\label{it:mp:uniformlyregular} $N=f^{-1}(0)$ and $0$ is a \emph{uniformly regular value} of $f$. By the latter we mean that $df$ is fiberwise invertible on the normal bundle $V$ of $N$ in $M$ and that the isomorphisms $df|_{V_x}$ as well as their inverses $(df|_{V_x})^{-1}$ are uniformly bounded in $x\in N$.
\item\label{it:mp:properness} $f$ is \emph{uniformly proper away from $N$}, that is, preimages of compact subsets are contained in controlled neighborhoods of $N$.
\end{enumerate}
Condition \ref{it:mp:uniformlyregular} implies that $df|_V$ provides a trivialization of the normal bundle $V$ and we equip it with the $\KO$-orientation coming from this trivialization.
\end{defn}

Note that we do not assume $f$ to be proper in the usual sense, i.e. the submanifold~$N$ may be non-compact.

\begin{rem}
Schick and Zadeh consider in \cite{schick_zadeh} only compact submanifolds $N$. In this case, the function $f$ only has to be proper, uniformly continuous, 
smooth near $N=f^{-1}(0)$ and $0$ has to be a regular value of $f$. Indeed, uniform properness will become properness in the usual sense under compactness of $N$ and the uniform boundedness of the derivatives in \ref{it:mp:smoothboundedder} and of the norms of the isomorphisms in \ref{it:mp:uniformlyregular} are automatically satisfied then.

In \cite[Lemma 1.2]{schick_zadeh} it is shown that these conditions are satisfied if there exist $r$-many separating hypersurfaces intersecting suitably with each other, therefore explaining the name \emph{multi-partitioned}. The components of the function $f$ are the signed distances from the hypersurfaces.
They actually do not mention that the resulting function is uniformly continuous, although it clearly is, but instead only the slightly weaker property that it is continuous and uniformly expansive, also called controlled (meaning that for every $R > 0$ exists $S > 0$ such that $d(x,y) \le R$ implies $d(f(x),f(y)) \le S$).
From a practical point of view, this is not a big difference, because if $M$ has bounded geometry, then it is straightforward to modify the function $f$ such that it becomes uniformly continuous while retaining the other properties.
\end{rem}

\begin{thm}\label{thm_multi_partitioned}
If $M$ is multi-partitioned by a submanifold $N$, then the Thom class of $N$ in $M$ lifts through the coarse co-assembly map away from $N$.

Consequently, if $M$ is spin and has bounded geometry in some $R$-neighbourhood of~$N$, then there is a wrong way map
\[
\K(C^*(M,N;\Cl_m)) \to \K(C^*(N \subset M;\Cl_{m,r}))
\]
mapping $\ind(\slashed D_M,N)$ to $i_*\ind(\slashed D_N)$.
\end{thm}

\begin{proof}
We are going to verify the conditions of \cref{thm_thom_lifts_application}.

First we check that $N$ is uniformly embedded into $M$. Since we assumed that $f$ is smooth on an $R$-neighborhood of $N$, we can use the implicit function theorem to find for each $x$ a radius $0<R_x<R$ and a function $h_x\colon T_xM\supset W_x\to V_x$ on some open subset $W_x\subset T_xM$ whose graph describes $B_{R_x}(x)\cap N$ in normal coordinates.
It is characterized by the equation
\[\forall y\in W_x\colon f(\exp_x(y,h_x(y))=0\]
and applying the chain rule shows us that the derivatives of $h_x$ can be expressed in terms of:
\begin{itemize}
\item the derivatives of the normal coordinates of $M$, which are uniformly bounded on $R$-balls around points of $N$ because we assumed that $M$ has bounded geometry in the $R$-neighborhood of $N$;
\item the derivatives of $f$, which we assumed to be bounded on the $R$-neighborhood of $N$;
\item the fiberwise inverses of the differential $d(f\circ\exp_x)$ in normal direction.
\end{itemize}
Note that if $R_x$ is chosen small enough, then the norm of the fiberwise inverse in the last point is $1$-close to the norm of $(df|_{V_x})^{-1}$, which in turn is uniformly bounded in $x$ by our uniform regularity assumption \ref{it:mp:uniformlyregular}. For these choices of $R_x$, we see that each of the derivatives of $h_x$ is equi-uniformly bounded.

Now, a lower bound on the maximal radii $R_x$ such that all of the above properties hold can again be estimated in terms of the derivatives of the normal coordinates and $f$ in the $R$-neighborhood and the bounds from \ref{it:mp:uniformlyregular}. Therefore, we may assume that all the $R_x$ are uniformly positive.
The functions $h_x$ then satisfy the conditions in \Cref{defn:boundedgeometry}.\ref{it:uniformlyembedded} except that they are not defined on all of $T_xM$. But this is not a problem, because one can easily modify and extend the functions outside of the open subsets $W'_x\subset W_x$ where the graph describes $N\cap B_{R_x}(x)$ such that their derivatives are equi-uniformly bounded everywhere. We have thus shown that $N$ is uniformly embedded into $M$.

It remains to prove that the Thom class lies in the image of the coarse co-assembly map away from $N$. 
The proof will become really simple under two further assumptions on the function $f$.
\begin{enumerate}\setcounter{enumi}{4}
\item\label{it:mpplus:smootheverywhere} We assume that $f$ is smooth with bounded gradient on all of $M$.
\end{enumerate}
This first additional assumption together with condition \ref{it:mp:properness} implies that we obtain a commutative diagram
\[
\xymatrix{
\K(\sHigCor_N(M;\Cl_{1,r})) \ar[d]^-{\mu_N} & \K(\sHigCor_{\{0\}}(\IR^r;\Cl_{1,r})) \ar[d]^-{\mu_{\{0\}}} \ar[l]_-{f^*} \\
\K(\cAz(N\Subset M;\Cl_{0,r})) & \K(\cAz(\{0\}\Subset \IR^r;\Cl_{0,r})) \ar[l]_-{f^*}
}
\]
where on the right hand side we have the co-assembly map $\mu_{\{0\}}$ away from $\{0\}$, which is just the usual coarse co-assembly map introduced in \cite{EmeMey}. We know that the latter is an isomorphism and therefore the Thom class of $0$ in $\R^r$ lifts through it.
Hence it suffices to show that the lower horizontal arrow in the diagram maps the Thom class of $0$ in $\R^r$ to the Thom class of $N$ in $M$ and to this end we make a second additional assumption.

\begin{enumerate}\setcounter{enumi}{5}
\item\label{it:mpplus:niceontubularneighborhood} Let $\Xi$ denote the diffeomorphism from \Cref{lem:boundedgeometryproperties} defined on some $R_1$-disk bundle in $V$.
We assume that there exists an $0<\varepsilon<R_1$ with the following property: Let $U\subset V$ be the unit $\varepsilon$-ball bundle and $W\subset M$ the $\varepsilon$-tubular neighborhood of $N$ in $M$. Then the composition 
\[U\xrightarrow{(\Xi^{-1},\pi)}W\times N\xrightarrow{f\times\id}\R^r\times N\]
extends to an isometric oriented bundle isomorphism $H=(h,\pi)$ between the normal bundle $V$ and the trivial rank $r$ real bundle.
Moreover, we assume that $f$ maps the complement of $W$ to the complement of the $\varepsilon$-ball in $\R^r$.
\end{enumerate}

Note that the bundle isomorphism $H$ is exactly the one which provided us with a trivialization of the normal bundle $V$ and hence with the canonical $\KO$-orientation. 
The additional assumption that $H$ is isometric implies that it induces an isomorphism 
\[G\colon S'_V\to \Cl_{0,r}\times N\]
between the $\Cl_{0,r}$-bundles associated to the canonical spin structures of the normal bundle and the trivial bundle. Since the derivatives of $\Xi^{-1}$ and $f$ are bounded on $U$ and $W$, respectively, the derivatives of $H$ and $G$ must also be bounded. Since these are isometric bundle isomorphisms, their inverses have the same property. We can therefore take exactly this $G$ in order to embed $S'_V$ as a direct summand of bounded variation into $(\Cl_{0,r}\grtensor\grelltwo)\times N$, cf.\ \Cref{lem_thom_N_bounded_geom}.
As a consequence, we obtain a diagram
\[\xymatrix@C=4ex{
\cS_\varepsilon\ar[r]\ar@{=}[dd]
&\Ct( D_\varepsilon^r, S_\varepsilon^{r-1};\Cl_{0,r})\ar[r]\ar[d]^{h^*}
&\cA( D_\varepsilon^r, S_\varepsilon^{r-1};\Cl_{0,r})\ar[r]\ar[d]^{h^*}
&\cAz(\{0\}\Subset\R^r;\Cl_{0,r})\ar[dd]^{f^*}
\\
&\Cub(\overline{U},\partial U;\Cl_{0,r})\ar[r]
&\cA(\overline{U},\partial U;\Cl_{0,r})
&
\\\cS_\varepsilon\ar[r]
&\Cub(\overline{U},\partial U;\pi^*\Lin_{\Cl_{0,r}}(S'_V))\ar[u]_{\Ad_{G}}
&\cA(\overline{W},\partial W;\Cl_{0,r})\ar[u]_{\Xi^*}\ar[r]
&\cA(N\Subset M;\Cl_{0,r})
}
\]
where $D_\varepsilon^r,S_\varepsilon^{r-1}$ denote the $r$-disk and $(r-1)$-sphere of radius $\varepsilon$ in $\R^r$ and $\Cub$ stands for uniformly continuous bounded functions. Note that on manifolds with bounded geometry, the uniformly continuous functions are exactly the norm closure of the smooth functions with bounded gradient, cf. \cite[Lemma 3.10]{engel_uniform_Kth}, so these function algebras indeed fit into the diagram. 
The two leftmost horizontal maps are the canonical $*$-homomorphisms $\varphi\mapsto\varphi(C)$ in their respective contexts such that composition along the top is a representative of the Thom class of $0$ in $\R^r$ and the composition along the bottom is the canonical representative of the Thom class of $N$ in $M$.
The diagram commutes because of conditions \ref{it:mp:properness}, \ref{it:mpplus:smootheverywhere}, and \ref{it:mpplus:niceontubularneighborhood}, as one readily checks, and hence the claim follows.

For general $f$ we now just need to prove that we can modify it such that the two additional assumptions hold.
First, we apply the polar decomposition fiberwise to write $df|_{V_x}=P_x\circ h(x)$ where $h(x)\colon V_x\to\R^n$ is an isometric isomorphism and $P_x\colon \R^n\to\R^n$ is a positive semidefinite linear map.
By the uniform regularity condition \ref{it:mp:uniformlyregular}, the polar decomposition is unique and there exist $c,C>0$ such that $c\cdot \id\leq P_x\leq C\cdot \id$ for all $x\in N$, without loss of generality $c\leq 1$. Note that therefore $P_x$ and $h(x)$ must depend smoothly on $x\in N$, because they could only be unsmooth (even discontinuous) at points where $P_x$ is not invertible.
Hence we obtain an isometric bundle isomorphism  $H=(h,\psi)\colon V\to\R^r\times N$ and a smooth map $P\colon \R^r\times N\to\R^n$ which is fiberwise positive semidefinite linear and such that $df|_V=P\circ H$.

Since the derivatives of $f$ and $\Xi^{-1}$ are bounded on an $R$-neighborhood of $N$, the difference $df\circ\Xi^{-1}-f$ is uniformly of quadratic order in the distance of $N$ and we see that there exists $\varepsilon>0$ such that $4\varepsilon c^{-1}<R$ and $\|df\circ\Xi^{-1}(x)-f(x)\|<\varepsilon$ for all $x$ in the $4\varepsilon c^{-1}$-neighborhood of $N$. This implies that 
\[\|f(x)\|\geq \|df\circ\Xi^{-1}(x)\|-\varepsilon\geq c\|H\circ \Xi^{-1}(x)\|-\varepsilon\geq    2\varepsilon\]
for all $x\in M$ whose distance from $N$ is between $3\varepsilon c^{-1}$ and $4\varepsilon c^{-1}$. 
In addition $f$ does not take on the value zero outside of the $4\varepsilon c^{-1}$-neighborhood of $N$, so we can easily modify $f$ such that $\|f(x)\|\geq 2\varepsilon$ for all $x$ outside of the $3\varepsilon c^{-1}$-neighborhood of $N$.
Now we can approximate $f$ by a smooth function $f_1$ with bounded gradient and $\|f-f_1\|_\infty<\varepsilon$, because $f$ is uniformly continuous. Choose a smooth function $\psi\colon M\to [0,1]$ which only depends on the distance from $N$, is one on the  $3\varepsilon c^{-1}$-neighborhood and vanishes outside of the $4\varepsilon c^{-1}$-neighborhood. We then define $f_2\coloneqq \psi\cdot h\circ\Xi^{-1}+(1-\psi)\cdot f_1$.
Clearly, $f_2$ is a smooth function with bounded gradient such that $(f\times\id)\circ(\Xi^{-1},\pi)$ is equal to the isometric bundle isomorphism $H=(h,\pi)$ on the $3\varepsilon c^{-1}$-neighborhood of $N$, which contains the $\varepsilon$-neighborhood. Moreover, for all $x\in M$ outside of the $\varepsilon$-neighborhood we have $\|f_2(x)\|\geq \varepsilon$. Indeed, if the distance of $x$ from $N$ is between $3\varepsilon c^{-1}$ and $4\varepsilon c^{-1}$, then 
\begin{align*}
\|f_2(x)\|&\geq \|\psi\cdot h\circ\Xi^{-1}(x)+(1-\psi)\cdot f(x)\|-\varepsilon
\\&\geq\|\psi\cdot h\circ\Xi^{-1}(x)+(1-\psi)\cdot P\circ H\circ \Xi^{-1}(x)\|-2\varepsilon
\\&\geq c\|h\circ \Xi^{-1}(x)\|-2\varepsilon\geq 3\varepsilon-2\varepsilon=\varepsilon
\end{align*}
and the other cases are even simpler.
It is also straightforward to see that $f_2$ satisfies all the other properties that $f$ had and hence we are done.
\end{proof}

As a special case of the above we will state the partitioned manifold index theorem, i.e.\,the case $r=1$ of \cref{thm_multi_partitioned}:
\begin{cor}\label{cor_partitioned}
Let $M$ be a complete Riemannian spin-manifold with a separating hypersurface $N$ and assume that $M$ has bounded geometry in some $R$-neighbourhood of $N$.

Then the Thom class of $N$ in $M$ lifts through the coarse co-assembly map away from $N$ and there is a map
\[
K_{*}(C^*(M,N)) \to K_{*-1}(C^*(N))
\]
mapping $\ind(\slashed D_M,N)$ to $i_* \ind(\slashed D_N)$.
\end{cor}
\begin{proof}
To show that $M$ is partitioned by $N$, we have to provide the map $f\colon M \to \IR$. In this case, we take $f$ to be the signed distance function to the submanifold $N$ (the sign depending on which side of $N$ we are in $M$).
\end{proof}

\printbibliography[heading=bibintoc]

\end{document}